\documentclass[11pt]{amsart}
\usepackage[utf8]{inputenc}
\usepackage[english]{babel}

\usepackage[T1]{fontenc}
\usepackage{amscd,amssymb,amsmath,amsthm,amstext,amsfonts}
\usepackage{graphicx}
\usepackage{subfigure,fancybox}
\usepackage{epic,eepic}
\usepackage{esint}
\usepackage[square,numbers]{natbib}
\usepackage{booktabs}
\usepackage{hyperref}
\usepackage{algorithm}
\usepackage{algpseudocode}
\usepackage{enumerate}  
\usepackage[capitalise,nameinlink]{cleveref}
\crefformat{equation}{\textup{#2(#1)#3}}
\crefname{figure}{Figure}{Figures}
\crefformat{subfigure}{Figure #2#1#3}
\usepackage{comment}
\usepackage{mathtools}
\usepackage{tikz,pgfplots}
\pgfplotsset{compat=newest}

\newtheorem{theorem}{Theorem}[section]

\newtheorem{lemma}[theorem]{Lemma}
\newtheorem{assumption}[theorem]{Assumption}

\newtheorem{conjecture}[theorem]{Conjecture}
\theoremstyle{definition}

\theoremstyle{remark}
\newtheorem{remark}[theorem]{Remark}
\numberwithin{theorem}{section}
\numberwithin{equation}{section}
\numberwithin{table}{section}
\numberwithin{figure}{section}

\allowdisplaybreaks

\newcommand{\R}{\mathbb R}
\newcommand{\N}{\mathbb N}
\newcommand{\mcT}{\mathcal{T}}
\newcommand{\mcP}{\mathbb{P}}
\newcommand{\mcA}{\mathcal{A}}
\newcommand{\mcV}{\mathcal{V}}

\newcommand{\norm}[2]{\left\|#1\right\|_{#2}}
\newcommand{\normeps}[1]{\left\|#1\right\|_{V,\varepsilon}}
\newcommand{\seminorm}[2]{\left|#1\right|_{#2}}
\newcommand{\Pe}{{\rm{Pe}}}
\newcommand{\inner}[3]{\left( #1,#2\right)_{#3}}
\newcommand{\dual}[3]{\langle #1,#2\rangle_{#3}}
\newcommand{\abs}[1]{\left|#1\right|}
\newcommand{\Div}{\operatorname{div}}

\newcommand{\proj}{\Pi}

\newcommand{\eps}{\varepsilon}
\newcommand{\D}{\: \mathrm{d}}

\newcommand{\GalSol}{u_H}
\newcommand{\PiGalSol}{u_H}

\newcommand{\PiSLODSol}{u_{H}^{\ell}}
\newcommand{\GalSpace}{V_H}

\newcommand{\SLODSpace}{V_{H}^{\ell}}
\newcommand{\GalTest}{v_H}

\newcommand{\PiSLODtest}{v_{H}^{\ell}}

\begin{document}
\title[SLOD for convection-dominated diffusion]{Super-localized orthogonal decomposition for convection-dominated diffusion problems}
\author[]{Francesca Bonizzoni$^\dagger$, Philip Freese$^\dagger$, Daniel Peterseim$^\ddagger$}
\address{${}^{\dagger}$ Institute of Mathematics, University of Augsburg, Universit\"atsstr.~12a, 86159 Augsburg, Germany}
\address{${}^{\ddagger}$ Institute of Mathematics \& Centre for Advanced Analytics and Predictive Sciences (CAAPS), University of Augsburg, Universit\"atsstr.~12a, 86159 Augsburg, Germany}
\email{\{francesca.bonizzoni, philip.freese, daniel.peterseim\}@uni-a.de}
\thanks{The work of all authors is part of a project that has received funding from the European Research Council ERC under the European Union's Horizon 2020 research and innovation program (Grant agreement No.~865751).}

\begin{abstract}
This paper presents a multi-scale method for convection-dominated diffusion problems in the regime of large P\'eclet numbers. The application of the solution operator to piecewise constant right-hand sides on some arbitrary coarse mesh defines a finite-dimensional coarse ansatz space with favorable approximation properties. For some relevant error measures, including the $L^2$-norm, the Galerkin projection onto this generalized finite element space even yields $\varepsilon$-independent error bounds, $\varepsilon$ being the singular perturbation parameter. By constructing an approximate local basis, the approach becomes a novel multi-scale method in the spirit of the Super-Localized Orthogonal Decomposition (SLOD). The error caused by basis localization can be estimated in an a-posteriori way. In contrast to existing multi-scale methods, numerical experiments indicate $\varepsilon$-independent convergence without preasymptotic effects even in the under-resolved regime of large mesh P\'eclet numbers. 
\end{abstract}

\maketitle

\vspace{1cm}
multi-scale method;
super-localization;
singularly perturbed
\\[2ex]
 	
\section{Introduction}

This paper studies the numerical solution of the following singularly perturbed con\-vec\-tion-diffusion problem in a bounded polygonal domain $\Omega\subset\R^d$ with dimension $d=1,2,3$. Given some small diffusivity $0<\varepsilon\ll 1$, an incompressible (divergence-free) and bounded velocity field $b$ as well as an external force $f$, we look for $u$ such that the boundary value problem
\begin{equation}
\label{eq:pde}
\left\{\begin{array}{rl}
-\varepsilon\Delta u+b\cdot \nabla u=f&\text{in }\Omega\\
u=0&\text{on }\partial\Omega
\end{array}\right.
\end{equation}
holds in suitably weak sense. 

This fairly simple model problem appears to be very challenging for classical Galerkin finite element methods (FEMs) and related schemes when the ratio of the convection rate over the diffusion is large, that is, for large P\'eclet number $\Pe=\norm{b}{L^\infty(\Omega)}\varepsilon^{-1}$. In this regime, the solution $u$ typically develops exponential and parabolic layers at the boundary (and possibly interior layers in the presence of inhomogeneous Dirichlet data). Unless the width $h$ of the FE mesh resolves the characteristic length scale $1/\Pe\approx \varepsilon$ of these layers, FE approximations show spurious oscillations. To avoid this unstable preasymptotic behavior, a minimal resolution condition of the form $h\Pe\lesssim 1$ is typically required. However, in many relevant practical applications, $\varepsilon$ may be so small that such conditions are unfeasible. 

The circumvention or at least relaxation of this resolution condition has been subject of intensive research in the past few decades. We refer to the monograph \cite{Roos-Stynes-Tobiska} for a detailed overview on the subject. Several branches of solution strategies have been developed. One is based on mesh refinement or grading toward the layers~\cite{Bakhvalov,Miller-ORiordan-Shishkin, Farrell-Hegarty-Miller-ORiordan-Shishkin,Melenk2002}.
The more popular alternative, in particular in the engineering communities, is the class of stabilized methods. Roughly speaking, these approaches change the model on the continuum or discrete level by adding artificial diffusion along the negative velocity field (upwinding). Among the extensive number of existing approaches in this context, we mention the streamline upwind/Petrov--Galerkin method~\cite{Brooks-Hughes} (also known as streamline diffusion method - see, e.g., \cite{Johnson}), the Galerkin least-squares method~\cite{Hughes-Franca-Hulbert}, the Douglas--Wang Galerkin method~\cite{Franca-Frey-Hughes}, discontinuous Petrov--Galerkin methods~\cite{Demkowicz-Gopalakrishnan-Niemi, Li-Demkowicz}, hybridizable discontinuous Galerkin methods~\cite{Qiu-Shi}, residual-free bubble methods~\cite{Brezzi-Marini-Suli,Cangiani-Suli}, nonconforming stabilized virtual element methods~\cite{Berrone-Borio-Manzini} and edge-based methods with additional nonlinear diffusion~\cite{Barrenechea-Burman-Karakatsani}. 

It has been observed that many of these stabilized schemes are strongly related to multi-scale methods, which mark a third class of approaches to tackle strong convection
\cite{Hughes-Sangalli}. The essential idea of multi-scale methods is to resolve the fine-scale features such as strong gradients in the layers by locally precomputed generalized FE shape functions. Prime examples are variational multi-scale methods (VMS)~\cite{Hughes-Feijoo-Mazzei-Quincy, Larson-Maalqvist, Volker-Songul-William,doi:10.1137/090775592}, multi-scale FEMs~\cite{PARK-HOU, degond_lozinski_muljadi_narski_2015}, multi-scale hybrid-mixed methods~\cite{Harder-Paredes-Valentin}, multi-scale discontinuous Galerkin methods~\cite{Kim-Wheeler, Chung-Leung}, multi-scale virtual element methods~\cite{Xie-Wang-Feng}, multi-scale stabilization methods~\cite{Calo-Chung-Efendiev-Leung,Chung-Efendiev-Leung}, stabilization procedures by means of sub-grid scale~\cite{Codina}, energy minimizing generalized multi-scale methods~\cite{Zhao-Chung}, 
or the multi-scale method for time-dependent convection-dominant problems recently proposed in~\cite{Simon-Behrens}.

Although many of the approaches mentioned so far have been empirically successful in  applications and certainly improved upon the stability of standard FEMs, $\varepsilon$-independent behavior is hardly observed for large mesh P\'eclet numbers $h\Pe\gg 1$. 

This statement also applies to the Localized Orthogonal Decomposition (LOD) method which originated from VMS and is often referred to as numerical homogenization (for an overview on the topic, see~\cite{Altmann-Henning-Peterseim,Maalqvist-Peterseim,Owhadi-Scovel}). On an ideal level, the methodology realizes a prescribed projection of the unknown solution onto a discrete space (other than the Ritz projection) and, hence, allows best-approximation results in suitable norms independent of the P\'eclet number. However, existing practical versions based on the localization of the fine-scale Green's function \cite{Hughes-Sangalli,Elfverson,Pet15LODreview} do suffer from strong convection. While for moderate mesh P\'eclet numbers exponential decay results of \cite{MalqvistPeterseim2014,KPY18}  for the fine-scale Green's function still apply, they deteriorate with increasing mesh P\'eclet numbers as outlined in \cite{Li-Peterseim-Schedensack}. This prevents the construction of a localized basis by means of fine-scale correctors and limits the practical relevance of the approach. 

An alternative localization strategy was recently proposed in ~\cite{Hauck-Peterseim}  for the pure diffusion problem and then extended to indefinite and non-hermitian problems in~\cite{Freese-Hauck-Peterseim}. As outlined  in~\cite{Altmann-Henning-Peterseim}, the LOD (and also the VMS) implicitly computes its problem-adapted ansatz space by applying the solution operator to some classical FE spaces on coarse meshes. For the specific choice of piecewise constants the coarse space is simply given by the span of functions $\mcA^{-1}\mathbf{1}_T$, $\mcA^{-1}$ denoting the solution operator and
$\mathbf{1}_T$ being the characteristic function of the element $T$ ranging into a coarse mesh $\mcT_H$. 
We refer to the Galerkin projection method on such ansatz space as ideal method.
The novel localization strategy aims to identify local linear combinations of characteristic functions in such a way that the spread of the response under the solution operator is minimized. Since for the diffusion model problem this strategy yields a super-exponentially decaying localization error (as compared to the exponentially decaying localization error in classical LOD) the resulting practical method is referred to as Super-Localized Orthogonal Decomposition (SLOD).

The present paper shows that the super-localization strategy is not merely an amplification of the fine-scale Green's function, but allows localization in applications where it has not been observed before. We generalize the SLOD methodology to convection-diffusion problems with large P\'eclet number.  

The SLOD approximation error comprises two contributions: the discretization error of the ideal method and the localization error. As such, the error analysis consists of two major steps. The key result to bound the first contribution is contained in \cref{lem:a_priori}, where a-priori estimates for the continuous convection-diffusion problem with linear velocity field are proved. Thanks to this result, $\varepsilon$-explicit (and in particular cases, even $\varepsilon$-independent) error upper bounds for the ideal method are derived. The second contribution, instead, is proved to be proportional to the computable quantity $\sigma$~\cref{eq:sigma}, which reflects the worst-case localization error.

Notably, the SLOD basis functions display an $\varepsilon$-independent behaviour. Indeed, as $\varepsilon$ gets smaller, they are not affected by oscillations nor their support increases (see \cref{fig:1D_basis} and \cref{fig:2d_basis} for a representation in the one- and two-dimensional frameworks). This represents a major improvement with respect to both classical LOD and the state-of-the-art multi-scale method in~\cite{Li-Peterseim-Schedensack}. From a practical point of view, this translates into significant computational savings, which in turn makes computations possible even in the three-dimensional framework (see \cref{sec:3D} for 3D numerical experiments).

The remainder of the paper is organized as follows. In \cref{sec:model_problem} a detailed description of the problem of interest in its variational formulation is shown, and a-priori upper bounds for the continuous solution of the convection-diffusion problem with affine velocity field are proven.
An ideal numerical homogenization method based on the $L^2$-orthogonal projection onto piecewise constants is introduced in \cref{sec:ideal_method}. The core of the paper are \cref{sec:super_localization,sec:SLOD}, where the novel localization approach is presented and turned into a practically feasible method. In \cref{sec:error_analysis} the error analysis is carried out.
\Cref{sec:numerical_implementation} explains the SLOD algorithm and in \cref{sec:numerical_experiments}  its performances are displayed by means of several two- and three-dimensional numerical experiments.

\section{Model problem}
\label{sec:model_problem}

Let $\Omega\subset\R^d$ be a polygonal domain with $d=1,2,3$, let $0<\varepsilon\le 1$ be a singular perturbation parameter and $b\in L^\infty(\Omega;\R^d)$ satisfy $\Div b=0$. Let $V\coloneqq H^1_0(\Omega)$ and define the bilinear form $a\colon V\times V\rightarrow\R$ by 
\begin{equation}
\label{eq:a}
a(u,v)\coloneqq \varepsilon\int_\Omega\nabla u\cdot\nabla v\, \D x+ \int_\Omega (b\cdot\nabla u) v\, \D x
\end{equation}
for all $u,\, v\in V$. Given some linear functional $F\in V'\coloneqq H^{-1}(\Omega)$ on $V$ then the weak formulation of the boundary value problem~\cref{eq:pde} seeks $u\in V$ such that, for all $v\in V$, 
\begin{equation}
\label{eq:pde_weak}
a(u,v)=F(v).
\end{equation}
From now on, we assume that the right-hand side is a bit more regular than minimal, i.e., it is of the form $F(\bullet)\coloneqq\inner{f}{\bullet}{L^2(\Omega)}$ for some $f\in L^2(\Omega)$. This additional regularity of the right-hand side will give rise to orders of approximations.
We focus on the convection-dominated regime, namely, $\varepsilon\ll 1$ and P\'eclet number $\Pe=\norm{b}{L^\infty(\Omega)}\varepsilon^{-1}\gg 1$.

\begin{remark}
	The method proposed below naturally applies to the case of non-constant diffusion coefficients, which may incorporate multi-scale features, i.e., the constant diffusivity $\varepsilon$ can be replaced by a variable one of the form $\eps A$ where $A\in L^\infty(\Omega;\R^{d\times d})$ is symmetric and positive definite almost everywhere in $\Omega$. Moreover, the method can be generalized to the case of convection-diffusion-reaction equations in a straight-forward way.
\end{remark}

Since $\Div b=0$, integration by parts implies, for all $v\in V$,
\begin{equation}\label{eq:wp}
a(v,v)=\varepsilon\int_\Omega\abs{\nabla v}^2\, \D x + \int_\Omega (b\cdot\nabla v) v\, \D x = \varepsilon \seminorm{v}{V}^2,
\end{equation}
where $\seminorm{\bullet}{V}=\norm{\nabla\bullet}{L^2(\Omega)}$ denotes the $H^1$-seminorm, which is a norm in $V$.
Moreover, for all $u,v\in V$, the application of Cauchy--Schwarz's and Poincar\'e's inequalities readily implies 
\begin{align}\label{eq:wp2}
a(u,v)&\leq C_a\seminorm{u}{V}\seminorm{v}{V},
\end{align}
for $C_a=C_a(\Omega,\norm{b}{L^\infty(\Omega)})=\varepsilon+C_P \norm{b}{L^\infty(\Omega)} >0$, where $C_P$ denotes the Poincar\'e constant.
By the Lax-Milgram theorem, the coercivity \cref{eq:wp} and the boundedness \cref{eq:wp2} show that Problem~\cref{eq:pde_weak} admits a unique solution $u\in V$ that satisfies the $\varepsilon$-dependent stability estimate
\begin{equation}
\label{eq:u_bound}
\seminorm{u}{V}\leq \frac{C_a}{\varepsilon}\norm{F}{H^{-1}(\Omega)}.
\end{equation}
For $F(\bullet)=(f,\bullet)_{L^2(\Omega)}$ and special velocities, the estimate can be sharpened. More importantly, in the weaker $L^2(\Omega)$-norm, even $\varepsilon$-independent stability results are possible. We refer to \cite[Lemma 2.1]{Eriksson-Johnson} which covers the special case $b = \begin{pmatrix}1 & 0 \end{pmatrix}^{\top}$. The subsequent lemma generalizes \cite[Lemma 2.1]{Eriksson-Johnson} to velocity fields fulfilling the following technical assumption:
\begin{assumption}
	\label{ass:b}
	The divergence-free velocity field $b$ is affine and such that for all $x\in\Omega$, $b(x) \neq 0$.
\end{assumption}
The result is phrased in the $\varepsilon$-scaled norm of $V$
\begin{equation}
\label{eq:normeps}
\normeps{\bullet}^2\coloneqq\varepsilon\seminorm{\bullet}{V}^2+\norm{\bullet}{L^2(\Omega)}^2,
\end{equation}
which is equivalent to the $\seminorm{\bullet}{V}$-norm for $\varepsilon\leq 1$, since for all $v\in V$ there holds
\begin{equation}
\label{eq:norm_equiv}
\sqrt{\varepsilon}\seminorm{v}{V}\leq\normeps{v}\leq \sqrt{1+C_P^2}\seminorm{v}{V}.
\end{equation}
\begin{lemma}
	\label{lem:a_priori}
	Let the constants $c_b,C_b>0$ be such that for all $x\in\Omega$, $c_b \leq \exp(-b\cdot x) \leq C_b$. Moreover, introduce $b_{\infty} \coloneqq \norm{2b-b(0)}{L^\infty(\Omega)}$ as well as $B_{\infty} \coloneqq \norm{b-b(0)}{L^\infty(\Omega)}$ and assume $\eps\leq 1 - \tfrac{B_{\infty}}{b_{\infty}}$. If the velocity field $b$ satisfies \cref{ass:b}, then the unique solution of \cref{eq:pde_weak} with $f\in L^2(\Omega)$ satisfies the estimate
	\begin{align*}
	\normeps{u} \leq \frac{C_{b}}{c_{b}\sqrt{b_{\infty}}\sqrt{(1-\eps) b_{\infty} - B_{\infty}} } \left(\frac{c_b}{C_b} + \frac{1}{b_{\infty} ((1-\eps) b_{\infty} - B_{\infty})}\right)^{1/2} \norm{f}{L^2(\Omega)}.
	\end{align*}
	In particular, for $\eps \leq \tfrac{1}{2}\left(1 - \tfrac{B_{\infty}}{b_{\infty}}\right)$, there holds
	\begin{align}
	\label{eq:stability}
	\normeps{u} \leq C_{stab} \norm{f}{L^2(\Omega)},
	\end{align}
	with $C_{stab}$ positive and $\varepsilon$-independent.
\end{lemma}
\begin{proof}
	First, we show the result for constant $b = b(0)$. In this case, we have $b_{\infty} = \abs{b}$ and $B_{\infty} = 0$. Consider the transformed dependent variable $v(x) = \exp(-b\cdot x) u(x)$, for all $x\in\Omega$. The first step is to derive the strong formulation for $v$ exploiting \cref{eq:pde}, which yields an equation of the form \cref{eq:pde} with right-hand side depending on $f$, $u$ and $v$. Thereafter, we multiply by $v$ to get
	\begin{align*}
	a(v,v) = \inner{\exp(-b\cdot \bullet) f}{v}{L^2(\Omega)} + 2 \eps \inner{\exp(-b\cdot \bullet) (b\cdot \nabla u)}{v}{L^2(\Omega)} - (\eps + 1) \seminorm{b}{}^2 \norm{v}{L^2(\Omega)}.
	\end{align*}
	Integrating by parts and using $\Div b = 0$, we get
	\begin{align*}
	2 \eps \inner{\exp(-b\cdot \bullet) (b\cdot \nabla u)}{v}{} = 2 \eps \seminorm{b}{}^2 \norm{v}{L^2(\Omega)},
	\end{align*}
	and consequently, using \cref{eq:wp}, we deduce
	\begin{align*}
	\eps \seminorm{v}{V}^2 \leq \norm{\exp(-b\cdot \bullet) f}{L^2(\Omega)} \norm{v}{L^2(\Omega)} + (\eps - 1) \seminorm{b}{}^2 \norm{v}{L^2(\Omega)}^2.
	\end{align*}
	Thus, we find
	\begin{align}\label{eq:a_priori_1}
	\eps \seminorm{v}{V}^2 + (1 - \eps) \seminorm{b}{}^2 \norm{v}{L^2(\Omega)}^2 \leq C_b \norm{f}{L^2(\Omega)} \norm{v}{L^2(\Omega)}.
	\end{align}
	This yields a bound on the $L^2(\Omega)$-norm of $v$ as
	\begin{align*}
	\norm{v}{L^2(\Omega)} \leq \frac{C_b}{(1-\eps) \seminorm{b}{}^2} \norm{f}{L^2(\Omega)}.
	\end{align*}
	Eventually, to bound the $L^2(\Omega)$-norm of the solution $u$, we use
	\begin{align}\label{eq:a_priori_L2}
	\norm{u}{L^2(\Omega)} \leq \frac{1}{c_b} \norm{\exp(-b\cdot \bullet) u}{L^2(\Omega)} = \frac{1}{c_b} \norm{v}{L^2(\Omega)}.
	\end{align} 
	The estimate on the $\seminorm{\bullet}{V}$-norm of the original solution $u$ follows by 
	\begin{align*}
	\eps\seminorm{u}{V}^2 &= a(u,u) = \inner{f}{u}{L^2(\Omega)} \leq \norm{f}{L^2(\Omega)} \norm{u}{L^2(\Omega)} \leq \frac{1}{c_b} \norm{f}{L^2(\Omega)} \norm{v}{L^2(\Omega)} \\
	&\leq \frac{C_b}{c_b (1-\eps) \seminorm{b}{}^2} \norm{f}{L^2(\Omega)}^2.
	\end{align*} 
	By combining the upper bounds on $\varepsilon\seminorm{u}{V}^2$ and $\norm{u}{L^2(\Omega)}$, we derive the desired estimate. For constant $b$, the estimate \cref{eq:stability} holds for $\eps \leq \tfrac{1}{2}$ with $C_{stab}=\sqrt{2}C_b(c_b|b|)^{-1}\sqrt{c_b C_b^{-1} + 2|b|^{-2}}$
	
	The case of an affine velocity fields $b$ follows the same lines. The equation corresponding to \cref{eq:a_priori_1} reads
	\begin{align*}
	\eps \seminorm{v}{V}^2 + \left((1 - \eps) b_{\infty} - B_{\infty}\right) b_{\infty} \norm{v}{L^2(\Omega)}^2 \leq C_b \norm{f}{L^2(\Omega)} \norm{v}{L^2(\Omega)}.
	\end{align*}
	From here, we proceed as for the constant case. In particular, the estimate~\cref{eq:stability} holds for the $\eps$-independent constant $C_{stab} = \tfrac{\sqrt{2}C_b}{c_b \sqrt{b_{\infty} - B_{\infty}}\sqrt{b_{\infty}}} \left(\tfrac{c_b}{C_b} + \tfrac{2}{b_\infty(b_{\infty} - B_{\infty})}\right)^{1/2}$.
\end{proof}

\begin{remark}
	\label{rem:reaction_term}
	For the case of a convection-diffusion-reaction equation, the result from \cref{lem:a_priori} is well known, but relies on the presence of the reaction term, see \cite[Lemma 1.18]{Roos-Stynes-Tobiska}. In this case, as well as the special convection-diffusion case with $b = \begin{pmatrix}
	1 & 0
	\end{pmatrix}^\top$ also (local) estimates on the directional derivative away from boundary layers are known, see \cite[Lemma 1.2]{Eriksson-Johnson} and \cite[Remark 1.19]{Roos-Stynes-Tobiska}.
\end{remark}

\section{An ideal multi-scale method}
\label{sec:ideal_method}

This section introduces an ideal multi-scale method that identifies an approximation of the solution $u$ in an operator-adapted ansatz space $V_H$, whose construction is based on some (possibly coarse) FE mesh. 

Let $\mcT_H$ be a (triangular or quadrilateral) shape-regular mesh of the domain $\Omega$, where $H$ denotes the global mesh size of $\mcT_H$, namely, $H=\max_{T\in\mcT_H} \operatorname{diam}(T)$. The degrees of freedom of the multi-scale method are associated with the mesh elements $T\in\mcT_H$ via the characteristic functions $\mathbf{1}_T$. Given the solution operator  $\mcA^{-1}\colon L^2(\Omega)\rightarrow V$ that maps each right-hand side function $f\in L^2(\Omega)$ to the corresponding unique weak solution of problem~\cref{eq:pde} and the standard FE space
\[
\mcP^0(\mcT_H)\coloneqq {\rm span}\left\{\mathbf{1}_T\, |\, T\in \mcT_H\right\}
\] 
of $\mcT_H$-piecewise constants, the finite-dimensional subspace $V_H\subset V$ is given by
\begin{equation}
\label{eq:VH}
\GalSpace\coloneqq\mcA^{-1}\mcP^0(\mcT_H)={\rm span}\left\{\mcA^{-1}\mathbf{1}_T\, |\, T\in\mcT_H\right\}.
\end{equation}
Note that we could have chosen FE spaces other than $\mcP^0(\mcT_H)$ for the approximation of the right-hand side. E.g. the paper \cite{Elfverson} considers discontinuous piecewise linears on simplicial meshes and \cite{doi:10.1137/090775592} considers continuous piecewise linears with zero boundary condition. More generally, a finite-dimensional space of linear functionals on $V$ could be considered. The authors in \cite{Li-Peterseim-Schedensack} implicitly use the Dirac delta functions $\delta_z$ for the interior vertices $z$ of $\mcT_H$. Clearly this is only possible in one dimension and requires regularization in higher dimensions. While in two dimensions this was somewhat justifiable, the three-dimensional case seemed not to be tractable with this choice.

Let $\proj_H\colon L^2(\Omega)\rightarrow \mcP^0(\mcT_H)$ denote the $L^2$-orthogonal projection operator and note that, for all $T\in\mcT_H$, $\proj_H v|_T$ is given by
\[
\proj_H v|_T=\frac{1}{|T|}\int_T v\, \D x.
\]
It is well-known that $\Pi_H$ fulfills the following local stability and approximation properties (see~\cite{Payne-Weinberger,Bebendorf})
\begin{align}
\label{eq:projHa}
\norm{\proj_H v}{L^2(T)}&\leq \norm{v}{L^2(T)}\quad\text{for all }\, v\in L^2(T),\\
\label{eq:projHb}
\norm{v-\proj_H v}{L^2(T)}&\leq \pi^{-1} H \norm{\nabla v}{L^2(T)}\quad\text{for all }\, v\in H^1(T).
\end{align}
Given the kernel $\mathcal W\coloneqq\ker(\Pi_H\vert_V)$ of $\Pi_H$ when restricted to $V$, $V_H$ is equivalently characterized as 
$$V_H = \{v_H\in V\,|\,\forall w\in \mathcal W: a(v_H,w)=0\}.$$
To see this, let $v_H=\mcA^{-1}p_H\in V_H$ with $p_H\in\mcP^0(\mcT_H)$ and $w\in \mathcal W$ and observe that
$$a(v_H,w)=a(\mcA^{-1}p_H,w)=\inner{p_H}{w}{L^2(\Omega)}=0.$$
This shows one inclusion and equality of the spaces follows by a dimensionality argument. (More details are found in \cite[Remark 3.7]{Altmann-Henning-Peterseim}). 
In the pure diffusion case this is the $a$-orthogonal complement of $\mathcal W$, which led the notion of orthogonal decomposition.

The concatenation of the $L^2$-orthogonal projection $\Pi_H$ and the solution operator $\mcA^{-1}$ defines an ideal multi-scale method that maps right-hand sides $f\in L^2(\Omega)$ onto $\GalSpace$. The resulting approximation $\PiGalSol\in\GalSpace$ is the unique function that satisfies, for all $\GalTest \in\GalSpace$,
\begin{equation}
\label{eq:PiGalProt}
a(\PiGalSol,\GalTest)=\inner{\Pi_H f}{\GalTest}{L^2(\Omega)}.
\end{equation}
Note that this is a non-standard projection onto the discrete space. It equals the Galerkin projection and the abstract Petrov--Galerkin framework of \cite{Altmann-Henning-Peterseim} only for $f\in \mcP^0(\mcT_H)$. For general $f\in L^2(\Omega)$ it differs from the more established variants. In the pure diffusion case it equals the collocation variant discussed in \cite{Hauck-Peterseim}.

In the following lemma we derive an $\varepsilon$-independent upper bound on the discretisation error under \cref{ass:b}.
\begin{lemma}
	\label{lem:galerkin_ideal}
	Let $f\in H^s(\Omega)$ with $s\in[0,1]$, and $b$ as in \cref{ass:b}. Denote with $u\in V$ and $\GalSol\in \GalSpace$ the unique solutions to~\cref{eq:pde_weak} and~\cref{eq:PiGalProt}, respectively. Then, there holds
	\begin{equation}
	\label{eq:galerkin_ideal}
	\normeps{u-\GalSol} \leq C_{stab}\norm{f-\Pi_H f}{L^2(\Omega)}
	\leq C\, C_{stab}H^s\norm{f}{H^s(\Omega)},
	\end{equation}
	where $C$, $C_{stab}$ are $\varepsilon$- and $H$-independent positive constants, $C_{stab}$ being introduced in \cref{lem:a_priori}.
\end{lemma}
\begin{proof}
	Since $u=\mcA^{-1}f$ and $\PiGalSol=\mcA^{-1}\Pi_H f$ we readily get
	\begin{align*}
	\normeps{u-\GalSol}
	&=\normeps{\mcA^{-1}f-\mcA^{-1}\Pi_H f}
	=\normeps{\mcA^{-1}(f-\Pi_H f)}.
	\end{align*}
	\Cref{lem:a_priori} provides an upper bound of the right-hand side. Altogether,  
	\[
	\normeps{u-\GalSol}
	\leq C_{stab}\norm{f-\Pi_H f}{L^2(\Omega)}
	\leq C\,C_{stab}H^s\norm{f}{H^s(\Omega)}, 
	\]
	where the last inequality holds for all right-hand sides $f\in H^s(\Omega)$ with $s\in[0,1]$.
\end{proof}

Apart the exactness of the ideal method for $f\in\mcP^0(\mcT_H)$, \cref{lem:galerkin_ideal} above contains an error bound in the weaker $L^2(\Omega)$-norm that is independent of $\varepsilon$. First order convergence is predicted without a preasymptotic regime. The numerical experiments of the later sections will rather report second order and even $\varepsilon$-independent first order for the $H^1(\Omega)$-seminorm. A more abstract version of the estimate of \cref{eq:galerkin_ideal} reads
\begin{equation*}
\norm{u-\GalSol}{Y} \leq \|\mcA^{-1}\|_{X\rightarrow Y}\norm{f-\Pi_H f}{X},
\end{equation*}
where $\|\mcA^{-1}\|_{X\rightarrow Y}$ refers to the norm of 
$\mcA^{-1}$ as a mapping between suitable Banach spaces $X$ and $Y$. Choosing  $X=H^{-1}(\Omega)$ and $Y=L^2(\Omega)$ or $Y=H^1(\Omega)$  or $Y=H^1(\omega)$ where $\omega\subset\Omega$ excludes the boundary layers would pave the way to proving the numerically observed rates. However, we are not aware of any $\varepsilon$-independent bounds of the required operator norms.

\section{Super-localization strategy}
\label{sec:super_localization}

The canonical basis functions $\{\mcA^{-1}\mathbf{1}_T\,|\, T\in\mcT_H\}$ of the operator-adapted approximation space $\GalSpace$ are non-local. To make the method practically feasible, localized basis functions have to be identified. The LOD provides a mechanism to construct an exponentially decaying basis that has been very successful in many applications. However, this is not the case when applied to convection-dominated problems, as we are interested here.
More precisely, when applying the abstract theory of \cite{Altmann-Henning-Peterseim} the exponential decay property deteriorates as $\varepsilon$ goes to $0$, and the error estimate of error committed by computing a localized approximation of the exponentially decaying basis is only shown to behave like $\varepsilon^{-1}H^{-1-d/2}\exp(-c\varepsilon \ell)$. This indicates that the localization parameter needs to grow algebraically in $\varepsilon^{-1}$ to make this quantity small. This is in line with practical experience, documented e.g. in \cite{Li-Peterseim-Schedensack}. Therein, the authors also discuss a possible improvement using anisotropic patches. However, the construction is based on point evaluation functionals and, hence, essentially limited to the one- and two-dimensional case. 

This section presents an advanced localization strategy, which has superior localization properties, yielding, in particular, super-exponential decay of the localization error. The main idea stays in the identification of local $\mcT_H$-piecewise constant source terms that yield rapidly decaying (or even local) responses under the solution operator $\mcA^{-1}$ of the convection-dominated problem~\cref{eq:pde}. This super-localization strategy, now known as the Super-Localized Orthogonal Decomposition (SLOD), has been first introduced in~\cite{Hauck-Peterseim} for the second order elliptic partial differential equation $-\Div(A\nabla u)=f$, and subsequently extended to indefinite non-hermitian problems in~\cite{Freese-Hauck-Peterseim}. 

For the subsequent derivation of the super-localization strategy, we need to introduce some notations. 
The local patch of level $\ell\in\N$ of a union of elements $S\subset \Omega$ is given by:
\begin{equation*}
N^\ell(S)\coloneqq \begin{cases}
\bigcup\{T\in\mcT_H\,|\, T\cap S\neq \emptyset\} & \ell=1\\
N^1(N^{\ell-1}(S)) & \ell=2,3,4,\ldots
\end{cases}
\end{equation*}
Let $\ell\in\N$ be fixed, such that no patch coincide with the entire domain $\Omega$. Given $T\in\mcT_H$, denote 
\begin{itemize}
	\item 
	$\omega\coloneqq N^\ell(T)$ its $\ell$-th order patch;
	\item
	$V_\omega\coloneqq \left\{ v|_\omega\hspace{1ex}\vert\, v\in V \right\}$ the restriction of $V$ to the patch $\omega$, equipped with the semi-norm $\seminorm{\bullet}{H^1(\omega)}$ and the norm $\norm{\bullet}{H^1(\omega)}$;
	\item
	$\mcT_{H,\omega}\coloneqq \{K\in\mcT_H\cap\omega\}$ the sub-mesh of $\mcT_H$ with elements in $\omega$;
	\item 
	$\proj_{H,\omega}\colon L^2(\Omega)\rightarrow\mcP^0(\mcT_{H,\omega})$ the $L^2$-orthogonal projection onto $\mcP^0(\mcT_{H,\omega})$.
\end{itemize}
Note that throughout the paper, we will not distinguish between functions in $H^1_0(\omega)$ and their $V$-conforming extension by $0$ to the full domain $\Omega$.

The (ideal) basis function $\varphi=\varphi_{T,\ell,\varepsilon}\in V_H$ associated with the element $T$ is given by the ansatz
\[
\varphi=\mcA^{-1} g\quad{\rm with}\quad g=g_{T,\ell,\varepsilon}\coloneqq \sum_{K\in \mcT_{H,\omega}} c_K\mathbf{1}_K,
\]
for some coefficients $(c_K)_{K\in\mcT_{H,\omega}}$ that will be determined afterwards.
In particular, $\varphi$ fulfills, for all $v\in V$,
\[
a(\varphi,v)=\inner{g}{v}{L^2(\omega)}.
\] 
The Galerkin projection of $\varphi$ onto the local subspace $H^1_0(\omega)$ is the function $\varphi^{\rm loc}=\varphi^{\rm loc}_{T,\ell,\varepsilon}\in H^1_0(\omega)$ satisfying, for all $v\in H^1_0(\omega)$,
\begin{equation}
\label{eq:phi_loc}
a_\omega(\varphi^{\rm loc},v)=\inner{g}{v}{L^2(\omega)},
\end{equation}
where $a_\omega(\cdot,\cdot)$ denotes the restriction of the bilinear form $a(\cdot,\cdot)$ to the subset $\omega$.
In general, the local function $\varphi^{\rm loc}$ is a poor approximation of the ideal function $\varphi$. Nevertheless, appropriate nontrivial choices of $g$ (i.e., of coefficients $(c_K)_{K\in\mcT_{H,\omega}}$) lead to highly accurate approximations in the energy norm.

Before stating the criterion for the choice of $g$, we need to recall a few results on traces of $V_\omega$-functions (see~\cite{Lions-Magenes} for more details).
Let $\gamma_0$ denote the trace operator on $\omega$ restricted to $V_\omega\subset V$ 
\begin{equation}
\label{eq:trace}
\gamma_0=\gamma_{0,\omega}\colon V_\omega\rightarrow H^{1/2}(\partial\omega),
\end{equation}
and let $X\coloneqq H^{1/2}(\partial\omega)$ denote its range. 
We define the normal derivative $\gamma_{\partial_n}u\in X^\prime$ of $u\in H^1(\omega)$ with $-\varepsilon\Delta u+b\cdot\nabla u\in L^2(\omega)$ as
\begin{align}
\dual{\gamma_{\partial_n}u}{v}{X^\prime\times X}
\coloneqq \frac{1}{\varepsilon}\left(-\inner{g}{v}{L^2(\Omega)} + a_\omega(\varphi^{\rm loc},v)\right). \label{def:normal_derivative}
\end{align}

Using the normal derivative, we may now characterize the localization error.

\begin{lemma}
	\label{lem:SLOD_loc_error}
	There holds:
	\begin{equation}
	\label{eq:loc_error}
	a(\varphi-\varphi^{\rm loc},v) 
	= -\eps \dual{\gamma_{\partial_n}\varphi^{\rm loc}}{\gamma_0v}{X^\prime\times X}
	\quad\text{for all }\, v\in V,
	\end{equation} 
	where $\gamma_{\partial_n}\varphi^{\rm loc}$ denotes the normal derivative of $\varphi^{\rm loc}$ as defined in \cref{def:normal_derivative}.
\end{lemma}
\begin{proof}
	Let $v\in V$. Then, there holds:
	\begin{align*}
	a(\varphi-\varphi^{\rm loc},v) 
	& = a(\varphi,v) - a(\varphi^{\rm loc},v)
	= \inner{g}{v}{L^2(\omega)} - a_\omega(\varphi^{\rm loc},v)\\
	& = - \varepsilon \dual{\gamma_{\partial_n}\varphi^{\rm loc}}{\gamma_0v}{X^\prime\times X}.
	\end{align*} 
\end{proof}

\begin{remark}
	In the previous works \cite{Hauck-Peterseim,Freese-Hauck-Peterseim}, the smallness of the normal derivative has been interpreted as the (almost) $L^2$-orthogonality of $g$ on the space of convection-harmonic functions. Here, however, we directly use the smallness of the normal derivative, which makes the algorithm even simpler and avoids the sampling of the respective space of convection-harmonic functions.
\end{remark}

From \cite[Theorem 31.30]{Ern-Guermond-2021-II}, under the assumption $\omega$ convex, we find that the local function $\varphi^{\rm loc}$ is in the space $H^2(\omega)$. Hence, from \cite[Example 4.16, Theorem 3.16]{Ern-Guermond-2021-I} we get that the normal derivative $\gamma_{\partial_n} \varphi^{\rm loc}$ is integrable, as $H^{1/2}(\partial\omega)$ is continuously embedded in $L^2(\partial \omega)$. Thus, since $v\in H^1(\omega)$, we may identify the dual pairing $\dual{\gamma_{\partial_n} \varphi^{\rm loc}}{\gamma_0 v}{X^\prime\times X}$ with the $L^2(\partial\omega)$-inner product $\int_{\partial \omega} \gamma_{\partial_n} \varphi^{\rm loc} \gamma_0 v \D s$.
The characterization in equation \cref{eq:loc_error} now yields the following estimate for the localization error:
\begin{align*}
a(\varphi-\varphi^{\rm loc},v) &= -\eps \int\limits_{\partial \omega} \gamma_{\partial_n} \varphi^{\rm loc} \gamma_0 v \D s \leq \eps \norm{\gamma_{\partial_n} \varphi^{\rm loc}}{L^2(\partial \omega)} \norm{\gamma_0 v}{L^2(\partial \omega)} \\
&\leq \eps \norm{\gamma_{\partial_n} \varphi^{\rm loc}}{L^2(\partial \omega)} C_{\gamma_0} \norm{v}{H^1(\omega)},
\end{align*}
where we used the boundedness of $\gamma_0\colon V_\omega\to L^2(\partial \omega)$ with constant $C_{\gamma_0}$. 

We conjecture the super-exponential decay in $\ell$ of the $L^2(\partial\omega)$-norm of the normal derivative, i.e.,  $\norm{\gamma_{\partial_n} \varphi^{\rm loc}}{L^2(\partial \omega)}$. This is justified by the numerical experiment shown in \cref{fig:decay_eigenvalue}, which displays the eigenvalues of the matrix given below in \cref{eq:normal_derivative_eigenvalue}.
\begin{figure}
	\input{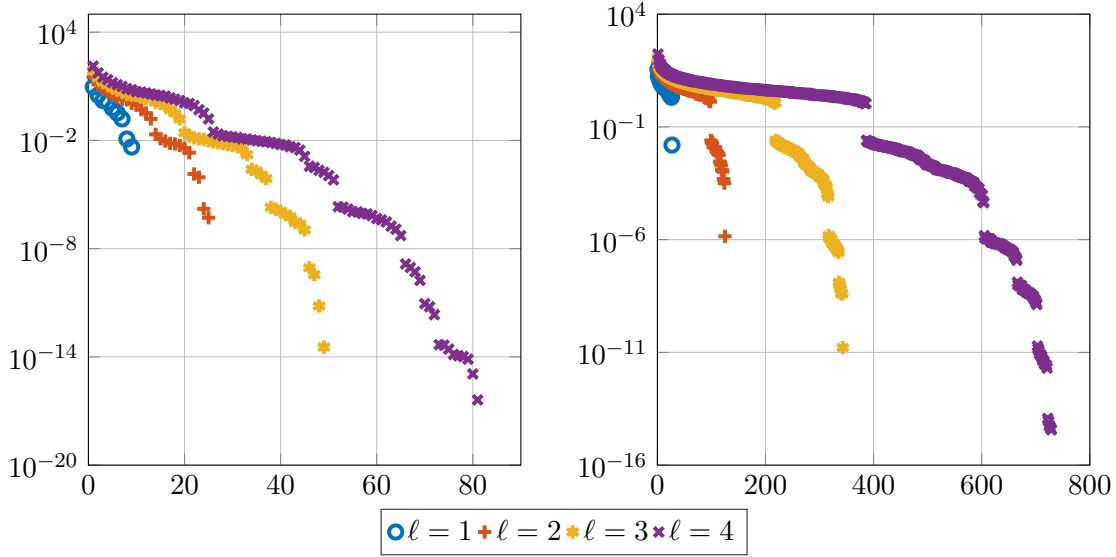}
	\caption{Decay of eigenvalues and respective normal derivative for a patch that does not reach the boundary. The velocity field $b$ is given as all ones. (a) Two-dimensional result for $\eps = 2^{-7}$ on a coarse mesh $H = 2^{-6}$. (b) Three-dimensional result for $\eps = 2^{-5}$ on a coarse mesh with $H=2^{-4}$.}
	\label{fig:decay_eigenvalue}
\end{figure}

\begin{conjecture}[Super-exponential decay]\label{con:sup_exp_decay}
	The quantity $\norm{\gamma_{\partial_n} \varphi^{\rm loc}}{L^2(\partial \omega)}$ decays super-ex\-ponentially in $\ell$, i.e., there exist constants $C_{sd}(\varepsilon,H,\ell)>0$ depending on $\varepsilon,\, H$ and $\ell$, but being independent of $T$, and $C>0$ independent of $\varepsilon,\, H,\, \ell$ and $T$ such that
	\begin{equation}
	\label{eq:conjecture}
	\norm{\gamma_{\partial_n} \varphi^{\rm loc}}{L^2(\partial \omega)} \leq C_{sd}(\varepsilon,H,\ell) \exp\left(-C\ell^{\frac{d}{d-1}}\right).
	\end{equation}
\end{conjecture}

\begin{remark}[SLOD basis in 1d]
	In the one-dimensional case, the boundary of the patches consists only of the two end points of the respective intervals, whereas we have three degrees of freedom for an order $\ell=1$ patch. Thus, the problem can be solved exactly, which yields a vanishing normal derivative on both end points of the patches. Hence, the \cref{con:sup_exp_decay} for $d=1$, interpreting $\frac{d}{d-1}$ as infinity, reveals a truly local basis function. This effect is also observed in \cref{fig:1D_basis}, where we compares three different basis functions in $V_{H}$ for various values of $\varepsilon$ and corresponding to the same mesh element $T\in\mcT_H$, namely $\mcA^{-1}\mathbf{1}_T$ (left); the basis function for $L^2$-projection based LOD (center); the SLOD basis function $\varphi^{\rm loc}_{T,1,\varepsilon}$ (right).
\end{remark}

\begin{figure}
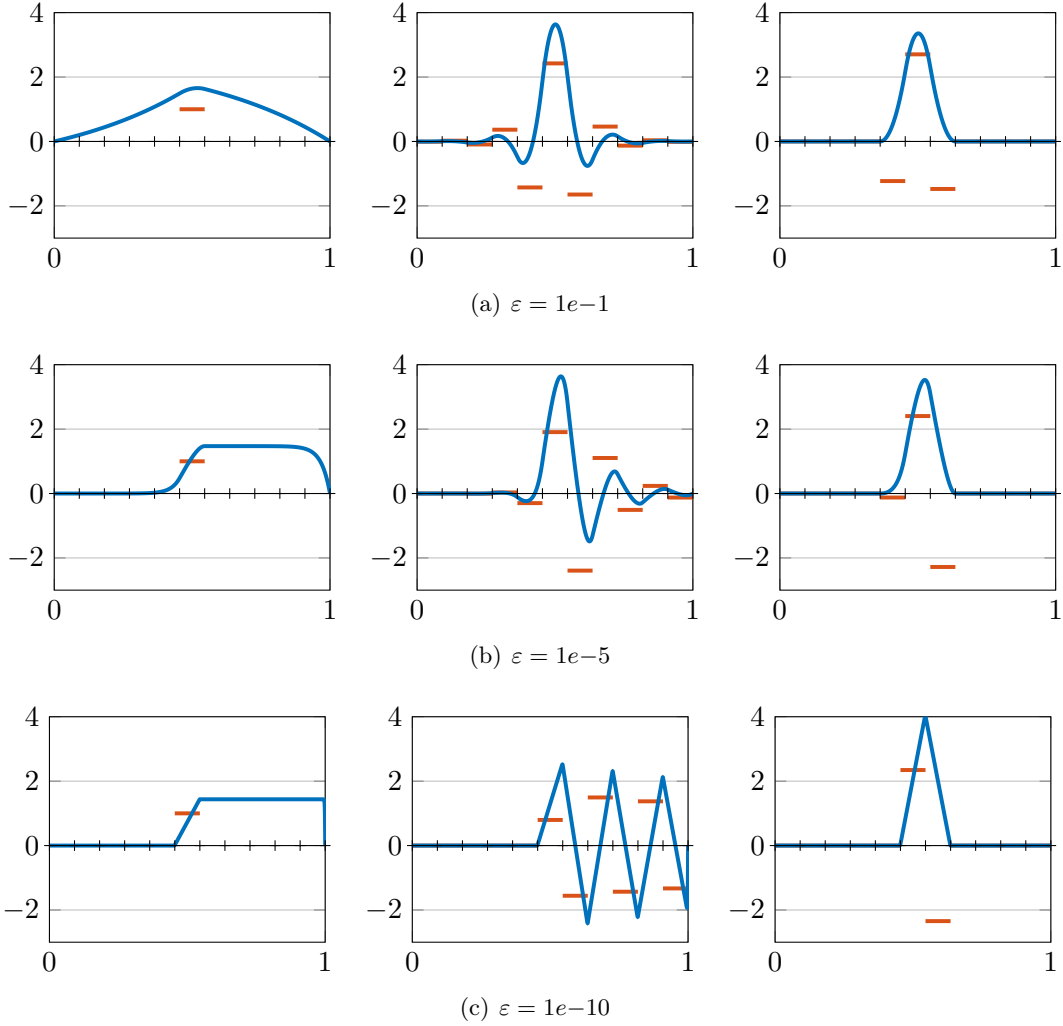

	\centering     
	\subfigure[$\varepsilon=1e{-1}$]{
		\input{basis_1D_ref_12_eps_1_l_4.tex}
		\label{fig:a}
	}
	\subfigure[$\varepsilon=1e{-5}$]{
		\input{basis_1D_ref_12_eps_5_l_4.tex}
		\label{fig:b}
	}
	\subfigure[$\varepsilon=1e{-10}$]{
%
%
\definecolor{mycolor1}{rgb}{0.85098,0.32549,0.09804}%
\definecolor{mycolor2}{rgb}{0.00000,0.44706,0.74118}%
\begin{tikzpicture}

\begin{axis}[%
width=3.666cm,
height=3cm,
at={(4.824cm,0cm)},
scale only axis,
xmin=0,
xmax=1,
xtick={0, 1},
ymin=-3,
ymax=4,
axis background/.style={fill=white},
xmajorgrids,
ymajorgrids
]
\addplot [color=mycolor1, line width=1.5pt, forget plot]
  table[row sep=crcr]{%
0	-3.88312437760163e-11\\
0.0909090909090909	-3.88312437760163e-11\\
};
\addplot [color=mycolor1, line width=1.5pt, forget plot]
  table[row sep=crcr]{%
0.0909090909090909	4.00950661560984e-11\\
0.181818181818182	4.00950661560984e-11\\
};
\addplot [color=mycolor1, line width=1.5pt, forget plot]
  table[row sep=crcr]{%
0.181818181818182	-5.10141373588624e-11\\
0.272727272727273	-5.10141373588624e-11\\
};
\addplot [color=mycolor1, line width=1.5pt, forget plot]
  table[row sep=crcr]{%
0.272727272727273	4.07114494893612e-08\\
0.363636363636364	4.07114494893612e-08\\
};
\addplot [color=mycolor1, line width=1.5pt, forget plot]
  table[row sep=crcr]{%
0.363636363636364	-0.000179877571140852\\
0.454545454545455	-0.000179877571140852\\
};
\addplot [color=mycolor1, line width=1.5pt, forget plot]
  table[row sep=crcr]{%
0.454545454545455	0.795614626180927\\
0.545454545454545	0.795614626180927\\
};
\addplot [color=mycolor1, line width=1.5pt, forget plot]
  table[row sep=crcr]{%
0.545454545454545	-1.55809916221376\\
0.636363636363636	-1.55809916221376\\
};
\addplot [color=mycolor1, line width=1.5pt, forget plot]
  table[row sep=crcr]{%
0.636363636363636	1.49390840898922\\
0.727272727272727	1.49390840898922\\
};
\addplot [color=mycolor1, line width=1.5pt, forget plot]
  table[row sep=crcr]{%
0.727272727272727	-1.43236219468693\\
0.818181818181818	-1.43236219468693\\
};
\addplot [color=mycolor1, line width=1.5pt, forget plot]
  table[row sep=crcr]{%
0.818181818181818	1.3733546341623\\
0.909090909090909	1.3733546341623\\
};
\addplot [color=mycolor1, line width=1.5pt, forget plot]
  table[row sep=crcr]{%
0.909090909090909	-1.3303373980189\\
1	-1.3303373980189\\
};
\addplot [color=mycolor2, line width=1.5pt, forget plot]
  table[row sep=crcr]{%
0	0\\
0.407293146306818	-0.000282726178803205\\
0.450062144886364	-0.000276091600018624\\
0.451282848011364	0.000407733007441902\\
0.452126242897728	0.00173204295492368\\
0.452858664772728	0.00429542024428464\\
0.453546697443182	0.00926985235082167\\
0.454212535511364	0.0188900818739461\\
0.455100319602273	0.0423508457521735\\
0.543168501420455	2.50316218533266\\
0.544145063920455	2.51714461491271\\
0.544411399147728	2.51796773830736\\
0.544699928977273	2.51648994423292\\
0.545077237215909	2.50943044130582\\
0.545587713068182	2.4867430713425\\
0.634477095170455	-2.37753687929378\\
0.635453657670455	-2.40500310641368\\
0.635697798295455	-2.4066588841798\\
0.635919744318182	-2.40529215004109\\
0.63623046875	-2.39737895840864\\
0.637184836647728	-2.34808108690644\\
0.725386186079545	2.2795868342706\\
0.726362748579545	2.30592150451737\\
0.726606889204545	2.30750906735296\\
0.726828835227273	2.30619864008555\\
0.727139559659091	2.29861145705046\\
0.728093927556818	2.25134456505487\\
0.816295276988636	-2.18567213367898\\
0.817271839488636	-2.21092183882674\\
0.817515980113636	-2.21244398526354\\
0.817737926136364	-2.21118753152217\\
0.818048650568182	-2.20391290106455\\
0.819003018465909	-2.15859321390332\\
0.907204367897728	2.09556865929916\\
0.908180930397728	2.11966203820456\\
0.908425071022728	2.12106929290727\\
0.908647017045455	2.11980439905739\\
0.908957741477273	2.11271851705042\\
0.909912109375	2.06881404430453\\
0.995028409090909	-1.90730815161338\\
0.996004971590909	-1.93031601351928\\
0.996249112215909	-1.93153052926977\\
0.996493252840909	-1.92983981685293\\
0.996826171875	-1.92132405142401\\
0.997292258522728	-1.89209719490812\\
0.997847123579545	-1.81514922319255\\
0.998446377840909	-1.64201456316566\\
0.999067826704545	-1.28081773065373\\
0.999711470170455	-0.53739525276899\\
1	0\\
};
\addplot [color=black, mark=+, mark options={solid, black}, forget plot]
  table[row sep=crcr]{%
0	0\\
0.0909090909090908	0\\
0.181818181818182	0\\
0.272727272727273	0\\
0.363636363636364	0\\
0.454545454545455	0\\
0.545454545454545	0\\
0.636363636363636	0\\
0.727272727272727	0\\
0.818181818181818	0\\
0.909090909090909	0\\
1	0\\
};
\end{axis}

\begin{axis}[%
width=3.666cm,
height=3cm,
at={(9.647cm,0cm)},
scale only axis,
unbounded coords=jump,
xmin=0,
xmax=1,
xtick={0, 1},
ymin=-3,
ymax=4,
axis background/.style={fill=white},
xmajorgrids,
ymajorgrids
]
\addplot [color=mycolor1, line width=1.5pt, forget plot]
  table[row sep=crcr]{%
0	0\\
0.0909090909090909	0\\
};
\addplot [color=mycolor1, line width=1.5pt, forget plot]
  table[row sep=crcr]{%
0.0909090909090909	0\\
0.181818181818182	0\\
};
\addplot [color=mycolor1, line width=1.5pt, forget plot]
  table[row sep=crcr]{%
0.181818181818182	0\\
0.272727272727273	0\\
};
\addplot [color=mycolor1, line width=1.5pt, forget plot]
  table[row sep=crcr]{%
0.272727272727273	0\\
0.363636363636364	0\\
};
\addplot [color=mycolor1, line width=1.5pt, forget plot]
  table[row sep=crcr]{%
0.363636363636364	-1.54026791321371e-12\\
0.454545454545455	-1.54026791321371e-12\\
};
\addplot [color=mycolor1, line width=1.5pt, forget plot]
  table[row sep=crcr]{%
0.454545454545455	2.34521117050717\\
0.545454545454545	2.34521117050717\\
};
\addplot [color=mycolor1, line width=1.5pt, forget plot]
  table[row sep=crcr]{%
0.545454545454545	-2.34520458931164\\
0.636363636363636	-2.34520458931164\\
};
\addplot [color=mycolor1, line width=1.5pt, forget plot]
  table[row sep=crcr]{%
0.636363636363636	0\\
0.727272727272727	0\\
};
\addplot [color=mycolor1, line width=1.5pt, forget plot]
  table[row sep=crcr]{%
0.727272727272727	0\\
0.818181818181818	0\\
};
\addplot [color=mycolor1, line width=1.5pt, forget plot]
  table[row sep=crcr]{%
0.818181818181818	0\\
0.909090909090909	0\\
};
\addplot [color=mycolor1, line width=1.5pt, forget plot]
  table[row sep=crcr]{%
0.909090909090909	0\\
1	0\\
};
\addplot [color=mycolor2, line width=1.5pt, forget plot]
  table[row sep=crcr]{%
0	0\\
0.449618252840909	0.000280931845535903\\
0.450838955965909	0.00098060125885624\\
0.451682350852272	0.00232584677362269\\
0.452414772727272	0.00492397663203725\\
0.453102805397728	0.00996114720564378\\
0.453768643465909	0.0196984465776691\\
0.454434481534091	0.0389542278151263\\
0.457120028409091	0.158699542996105\\
0.543301669034091	4.0005124396671\\
nan	nan\\
0.545876242897728	4.00022053102404\\
0.635142933238637	0.0234217392576186\\
0.636119495738637	0.0012590871927225\\
0.636541193181818	0\\
1	0\\
};
\addplot [color=black, mark=+, mark options={solid, black}, forget plot]
  table[row sep=crcr]{%
0	0\\
0.0909090909090908	0\\
0.181818181818182	0\\
0.272727272727273	0\\
0.363636363636364	0\\
0.454545454545455	0\\
0.545454545454545	0\\
0.636363636363636	0\\
0.727272727272727	0\\
0.818181818181818	0\\
0.909090909090909	0\\
1	0\\
};
\end{axis}

\begin{axis}[%
width=3.666cm,
height=3cm,
at={(0cm,0cm)},
scale only axis,
xmin=0,
xmax=1,
xtick={0, 1},
ymin=-3,
ymax=4,
axis background/.style={fill=white},
xmajorgrids,
ymajorgrids
]
\addplot [color=mycolor1, line width=1.5pt, forget plot]
  table[row sep=crcr]{%
0.454545454545455	1\\
0.545454545454545	1\\
};
\addplot [color=mycolor2, line width=1.5pt, forget plot]
  table[row sep=crcr]{%
0	0\\
0.450639204545455	0.000282657449923285\\
0.451859907670455	0.000986624527827162\\
0.452703302556818	0.00234013310778347\\
0.453435724431818	0.00495422177529359\\
0.454123757102273	0.0100223327744728\\
0.454900568181818	0.0210480428416364\\
0.544233842329545	1.42859888689814\\
0.545210404829545	1.43643530725674\\
0.545654296875	1.43687972464108\\
0.991654829545455	1.43660045924273\\
0.992875532670455	1.43590494053554\\
0.993718927556818	1.43456767553139\\
0.994451349431818	1.43198495877274\\
0.995139382102273	1.42697767062131\\
0.995805220170455	1.41729813663317\\
0.996471058238636	1.39815658805458\\
0.997136896306818	1.36030364157622\\
0.997802734375	1.28544838359043\\
0.998468572443182	1.13742002201152\\
0.999134410511364	0.844689806107216\\
0.999800248579546	0.265807635221778\\
1	0\\
};
\addplot [color=black, mark=+, mark options={solid, black}, forget plot]
  table[row sep=crcr]{%
0	0\\
0.0909090909090908	0\\
0.181818181818182	0\\
0.272727272727273	0\\
0.363636363636364	0\\
0.454545454545455	0\\
0.545454545454545	0\\
0.636363636363636	0\\
0.727272727272727	0\\
0.818181818181818	0\\
0.909090909090909	0\\
1	0\\
};
\end{axis}
\end{tikzpicture}%
		\label{fig:c}
	}
	\caption{Solution to the convection-dominated problem with right-hand side $\mathbf{1}_T$, i.e., $\mcA^{-1}\mathbf{1}_T$ (left); 
		$L^2$-projection based (global) LOD basis function (center); Novel SLOD basis function (right). Their corresponding $L^2$-normalized right-hand sides are depicted in orange.}
	\label{fig:1D_basis}
\end{figure}

\begin{remark}[SLOD basis in 2d and 3d]
	While in the one-dimensional setting we were able to retrieve truly local basis functions, this is no longer true in higher dimensions. In \cref{fig:2d_basis} we depict the basis functions $\varphi^{\rm loc}_{T,4,\varepsilon}$ for an element $T$ whose patch does not reach the global boundary for $\eps = 2^{-7}$ and $\eps = 2^{-9}$. The velocity field $b$ is given as $b = \frac{1}{\sqrt{2}}\begin{pmatrix}
	1 & 1
	\end{pmatrix}^\top$. Moreover, the figure shows the response of the solution operator to the indicator function $\mathbf{1}_{T}$ that corresponds to $T$. It is clearly visible, that the SLOD basis functions decay very fast, especially in comparison to the ideal basis functions of the space \cref{eq:VH}.
\end{remark}

\begin{figure}
	\centering
	\subfigure[$\eps = 2^{-7}$]{
		\includegraphics[width=.35\linewidth]{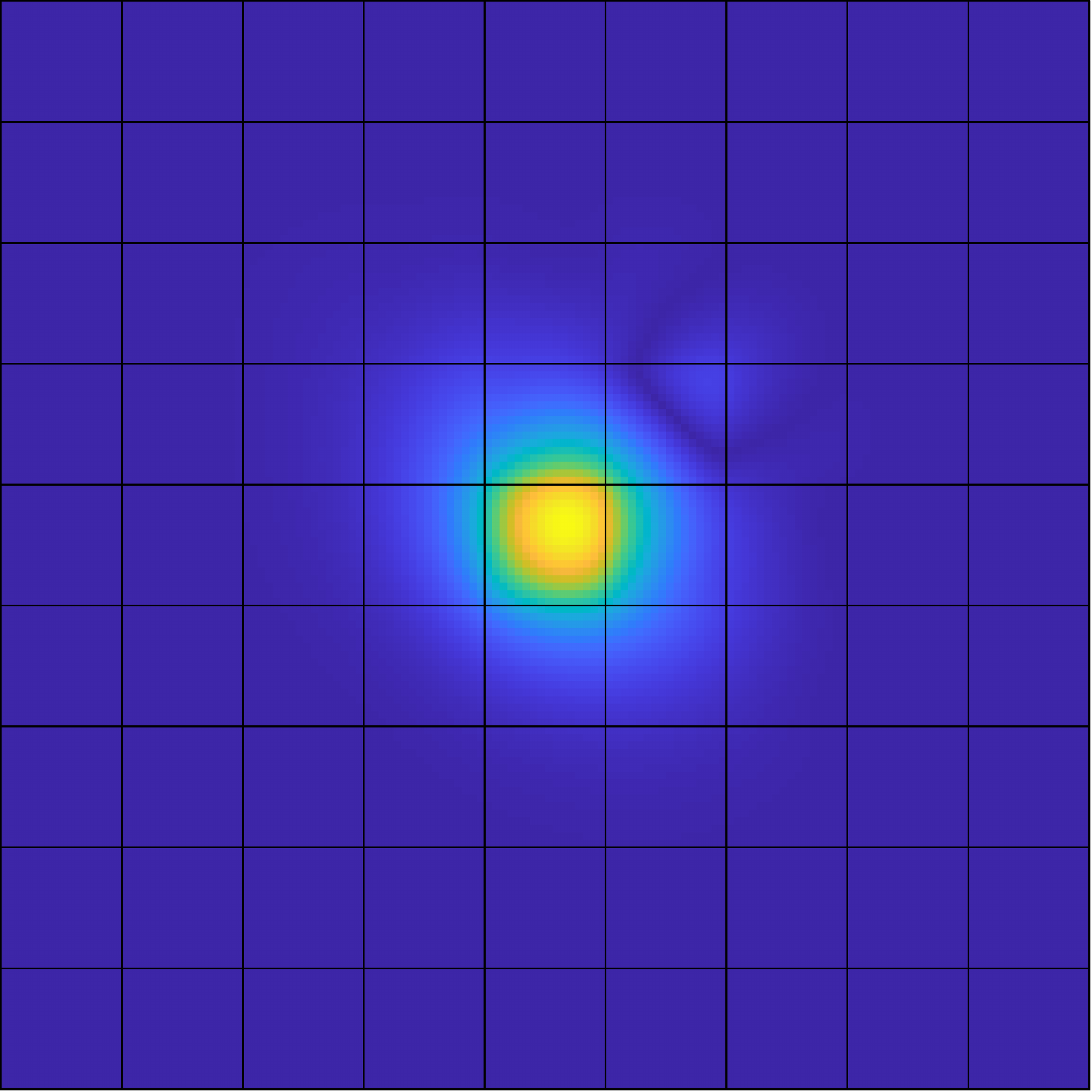}
		\includegraphics[width=.35\linewidth]{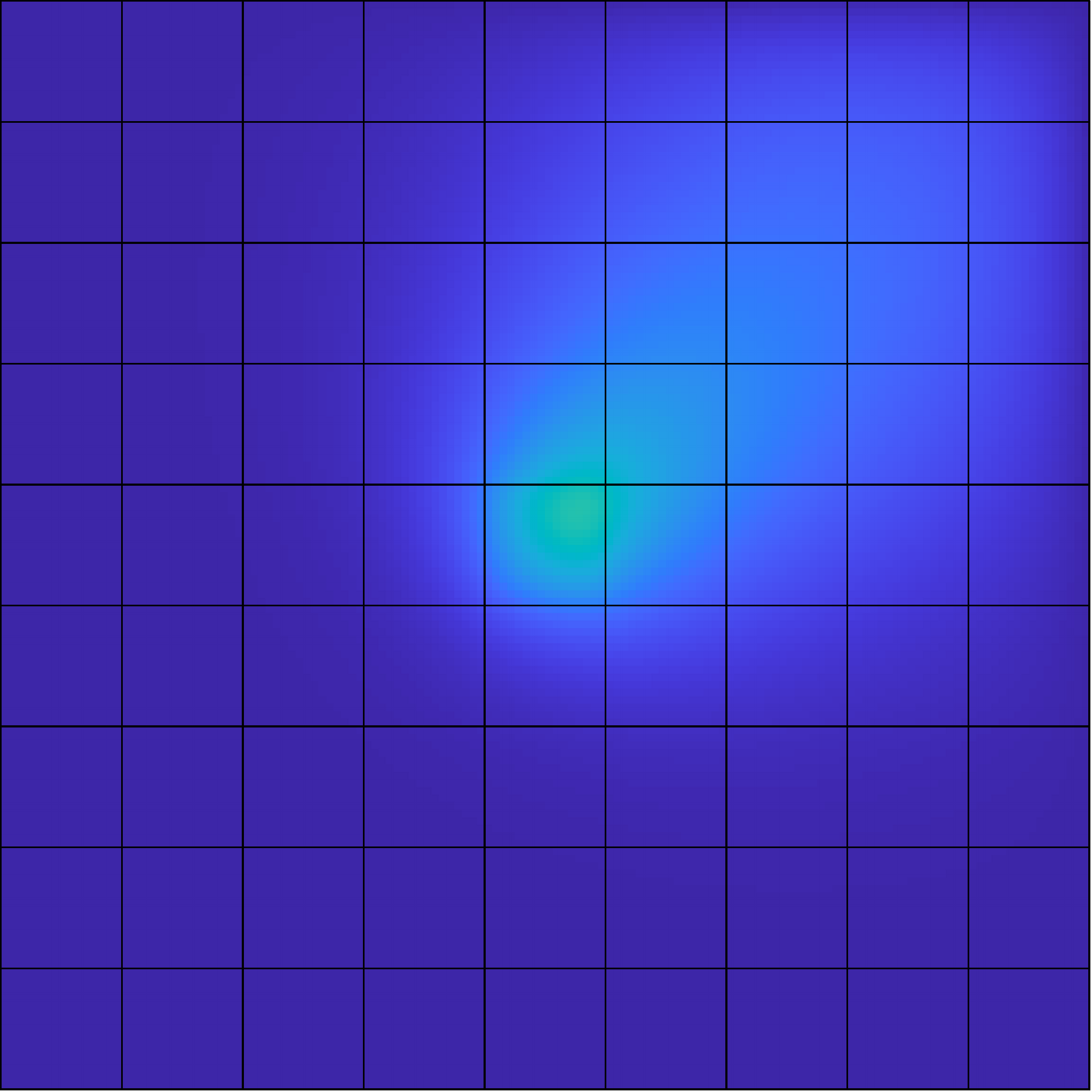}
	}
	\subfigure[$\eps = 2^{-9}$]{
		\includegraphics[width=.35\linewidth]{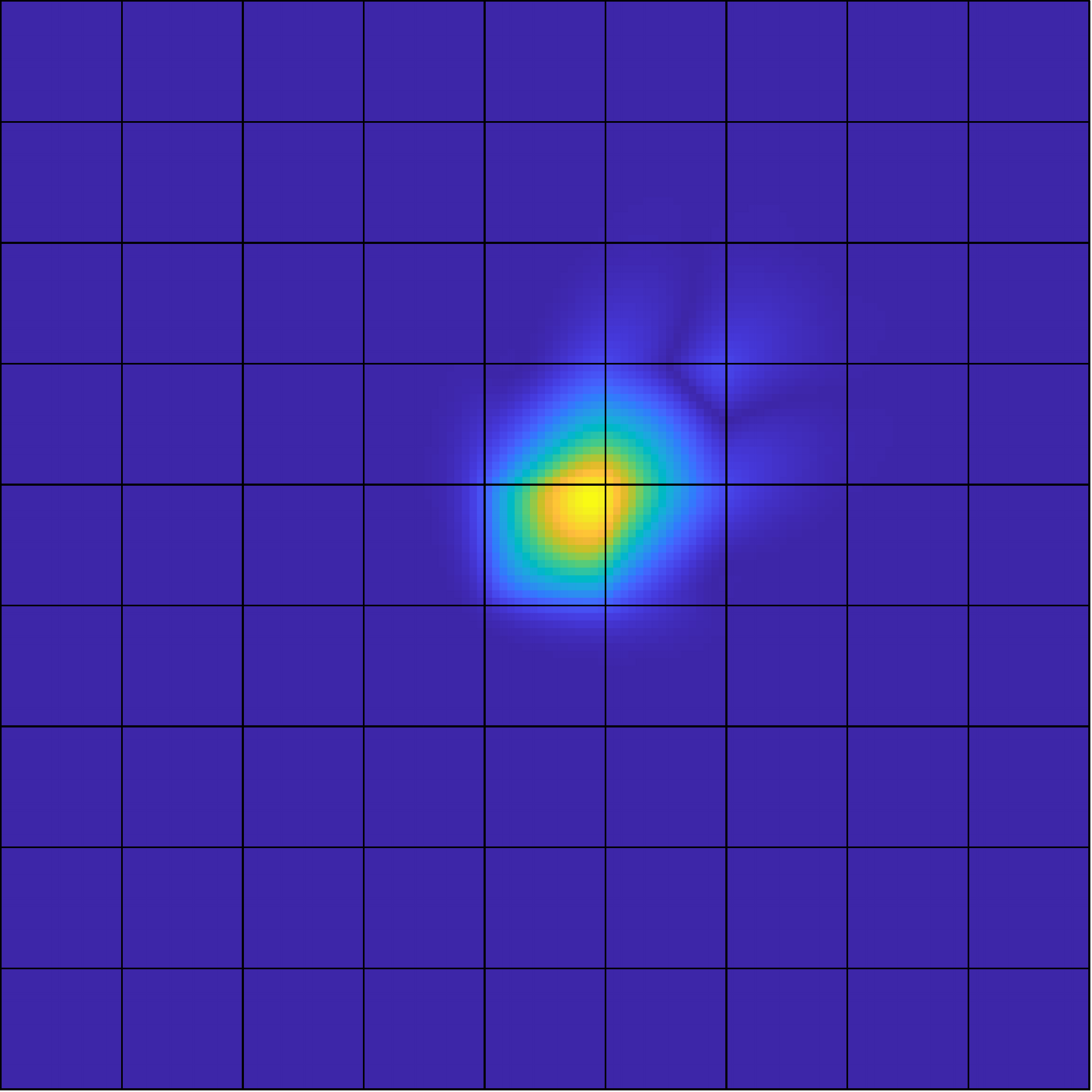}
		\includegraphics[width=.35\linewidth]{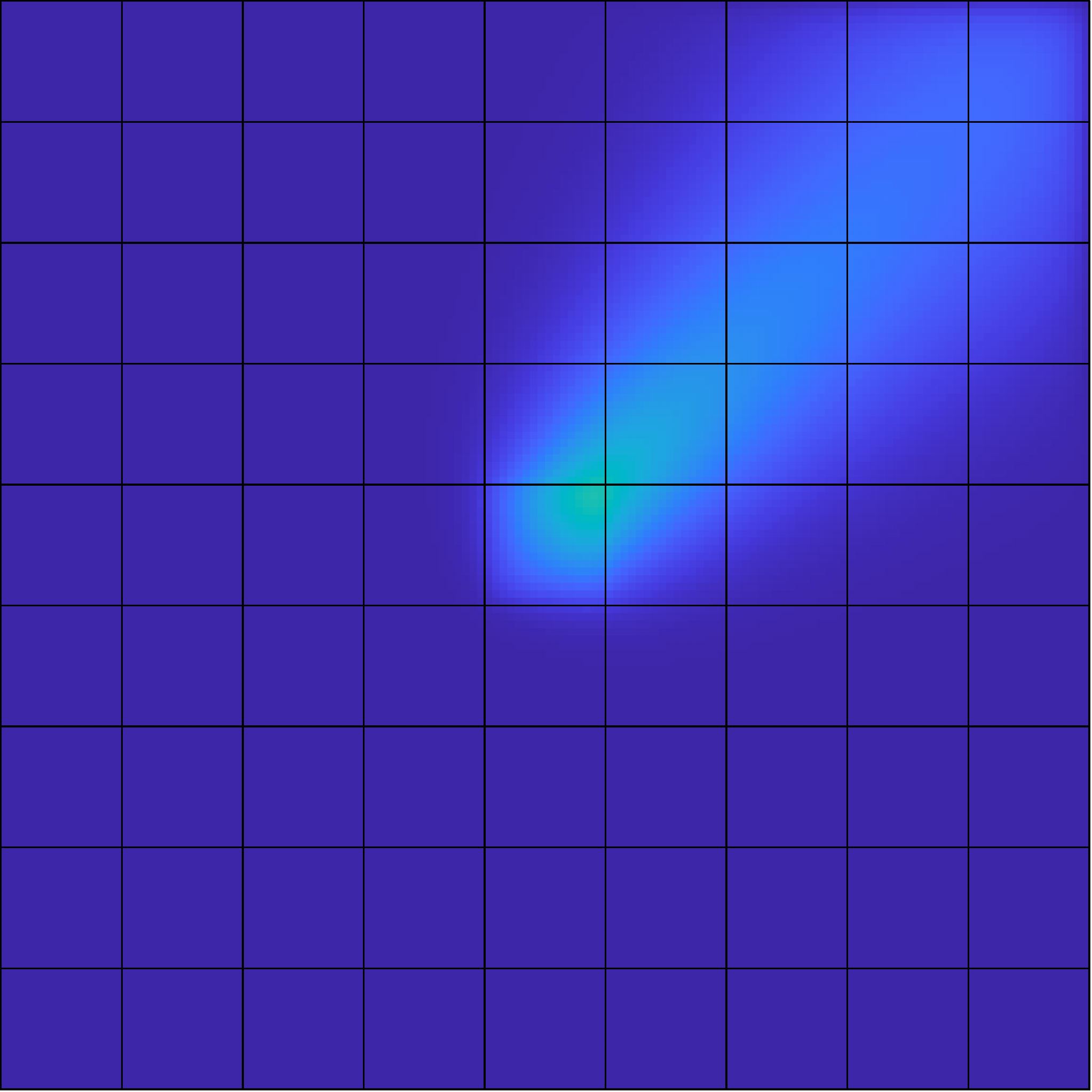}
	}	
	\caption{Absolute value of SLOD basis (left) and solution of $\mathcal{A}^{-1}\mathbf{1}_T$ (right) on $4$-th order (interior) patches, for $b = \frac{1}{\sqrt{2}}\begin{pmatrix}
		1 & 1
		\end{pmatrix}^\top$. 
	}
	\label{fig:2d_basis}
\end{figure}

\section{Super-localized multi-scale method}
\label{sec:SLOD}

Within this section we turn the method \cref{eq:PiGalProt} based on the ideal operator-adapted ansatz subspace $\GalSpace\subset V$ into a feasible numerical scheme, by means of the super-localization strategy introduced above.

Let the oversampling parameter $\ell$ be fixed. We define the ansatz space of the super-localized method as the span of the SLOD basis functions $\varphi_{T,\ell,\varepsilon}^{\rm loc}$ as $T$ varies in the coarse grid $\mcT_H$, namely:
\begin{equation}
\label{eq:VHell}
\SLODSpace \coloneqq \left\{\varphi_{T,\ell,\varepsilon}^{\rm loc}\,|\,T\in\mcT_H\right\}\subset V.
\end{equation}
The approximate solution provided by the SLOD method is the Galerkin projection in the space $\SLODSpace$ of the convection-dominated problem at hand with perturbed right-hand side $\Pi_H f$. In particular, the SLOD approximation to~\cref{eq:pde_weak} is the function $\PiSLODSol\in\SLODSpace$ such that, for all $\PiSLODtest\in \SLODSpace$,
\begin{equation}
\label{eq:SLOD_method}
a(\PiSLODSol,\PiSLODtest)=\inner{\Pi_H f}{\PiSLODtest}{L^2(\Omega)}.
\end{equation}

\begin{remark}[Collocation version]
	\label{rem:collocation}
	Expanding $\proj_Hf\in\mcP^0(\mcT_H)$ in the basis $\left\{g_{T,\ell,\varepsilon}\,|\,T\in\mcT_H\right\}$ 
	\begin{equation}
	\label{eq:f_coll}
	\proj_Hf=\sum_{T\in\mcT_H}c_T g_{T,\ell,\varepsilon},
	\end{equation}
	we derive an alternative discrete approximation to $u\in V$ 
	\begin{equation}
	\label{eq:u_coll}
	u^c_{H,\ell}\coloneqq \sum_{T\in\mcT_H}c_T \varphi^{\rm loc}_{T,\ell,\varepsilon}.
	\end{equation}
	For the calculation of the solution this approach seems to be very promising, as the computation only involves linear combinations of known quantities and there is no need to actually build a new FEM using the space $V_{H,\ell}$. Apparently, the condition number of the corresponding matrix, which is build using the right-hand sides $g_{T,\ell,\varepsilon}$ is poor in comparison to that of the Galerkin approach. Hence, the computation using the collocation method suffers from ill-conditioning and the method is outperformed by a classical Galerkin scheme.
\end{remark}

\section{Error analysis}
\label{sec:error_analysis}

A minimal requirement for the stability and convergence of the Galerkin method \cref{eq:SLOD_method} and its Collocation variant \cref{eq:u_coll} is that the set of functions $\left\{g_{T,\ell,\varepsilon}\,|\,T\in\mcT_H \right\}$ spans $\mcP^0(\mcT_H)$ in a stable way. Numerically, this is ensured as described in \cref{sec:numerical_implementation}. For the subsequent theoretical analysis, we make the following assumption.
\begin{assumption}
	\label{ass:g_basis}
	The set $\left\{g_{T,\ell,\varepsilon}\,|\,T\in\mcT_H\right\}$ is a Riesz basis of $\mcP^0(\mcT_H)$, i.e., there exists a constant $C_{rb}(\varepsilon,H,\ell)$, depending only polynomially on $H$ and $\ell$, such that 
	\begin{equation}
	\label{eq:g_basis}
	C_{rb}^{-1}(\varepsilon,H,\ell)\sum_{T\in\mcT_H}|c_T|^2
	\leq \norm{\sum_{T\in\mcT_H}c_T g_{T,\ell,\varepsilon}}{L^2(\Omega)}^2
	\leq C_{rb}(\varepsilon,H,\ell)\sum_{T\in\mcT_H}|c_T|^2.
	\end{equation}
\end{assumption}

In the following theorem we derive an a-priori error estimate for the solution to problem~\cref{eq:SLOD_method}. The upper bound is explicit in the quantity
\begin{equation}
\label{eq:sigma}
\sigma(\varepsilon,H,\ell)\coloneqq \max_{T\in\mcT_H} \norm{\gamma_{\partial_n} \varphi^{\rm loc}_{T,\ell,\varepsilon}}{L^2(\partial\omega)} 
\end{equation}
which reflects the worst-case localization error.

\begin{remark}[Exponential decay of classical LOD]
	For moderate mesh P\'eclet number, the quantity $\sigma(\varepsilon,H,\ell)$ in~\cref{eq:sigma} decays exponentially in the oversampling parameter $\ell$ (see~\cite[Appendix A]{Hauck-Peterseim} for the proof in the pure diffusion case). In particular, one can recover the a-priori error estimate with rates as for the LOD theory as in~\cite{Altmann-Henning-Peterseim,Li-Peterseim-Schedensack,Elfverson}. 
\end{remark}

\begin{theorem}[Convergence of the SLOD method]
	\label{thm:SLOD_conv}
	Let \cref{ass:g_basis} and \cref{ass:b} be satisfied. Then, there exists a constant $C>0$ independent of $H,\,\ell,\,\varepsilon$ such that, for all $f\in H^s(\Omega)$ with $s\in[0,1]$, there holds
	\begin{align}
	\label{eq:SLOD_conv}
	\normeps{u-\PiSLODSol}
	& \leq C\left(C_{stab}\norm{f-\Pi_H f}{L^2(\Omega)}
	+\varepsilon^{-1} \sigma(\varepsilon,H,\ell) C_{rb}(\varepsilon,H,\ell)^{1/2}\ell^{d/2}\norm{f}{L^2(\Omega)}  \right)\\
	\nonumber
	&\leq C\left(H^s\norm{f}{H^s(\Omega)}
	+\varepsilon^{-1} \sigma(\varepsilon,H,\ell) C_{rb}(\varepsilon,H,\ell)^{1/2}\ell^{d/2}\norm{f}{L^2(\Omega)}  \right),
	\end{align}
	where $C_{rb}(\varepsilon,H,\ell)$ and $\sigma(\varepsilon,H,\ell)$ are defined in \cref{ass:g_basis} and \cref{eq:sigma}, respectively.
\end{theorem}
\begin{proof}
	By triangular inequality, we get:
	\begin{equation}
	\label{eq:SLOD_proof}
	\normeps{u-\PiSLODSol}\leq 
	\normeps{u-\PiGalSol} + \normeps{\PiGalSol-\PiSLODSol}.
	\end{equation}
	The first term in~\cref{eq:SLOD_proof} represents the discretization error of the ideal multi-scale method, and its upper bound is given by~\cref{lem:galerkin_ideal}. We consider now the second term in~\cref{eq:SLOD_proof}, which represents the localization error. 
	Observe that $\PiGalSol$ solves the continuous equation for right-hand side $\Pi_Hf$. As a consequence, the SLOD solution $\PiSLODSol$ is the Galerkin approximation of $\PiGalSol$ in the finite dimensional space $\SLODSpace$. Using the norm equivalence~\cref{eq:norm_equiv} and applying C\'ea's Lemma, we get
	\begin{equation*}
	\normeps{\PiGalSol-\PiSLODSol}
	\lesssim \seminorm{\PiGalSol-\PiSLODSol}{V}
	\lesssim\frac{1}{\varepsilon} \inf_{\PiSLODtest\in\SLODSpace}\seminorm{\PiGalSol-\PiSLODtest}{V},
	\end{equation*}
	where the notation $x\lesssim y$ means $x\leq c y$ with $c$ positive constant independent of the mesh size parameter $H$, the localization parameter $\ell$ and the diffusion coefficient $\varepsilon$.
	Given the expansion of $\proj_Hf$ in the basis $\left\{g_{T,\ell,\varepsilon}\,|\,T\in\mcT_H\right\}$, namely, $\proj_Hf=\sum_{T\in\mcT_H}c_T g_{T,\ell,\varepsilon}$, we can express $\PiGalSol$ as 
	\begin{equation*}
	\PiGalSol=\sum_{T\in\mcT_H}c_T\mcA^{-1}g_{T,\ell,\varepsilon}
	=\sum_{T\in\mcT_H}c_T\varphi_{T,\ell,\varepsilon}.
	\end{equation*} 
	For the particular choice $\PiSLODtest=\sum_{T\in\mcT_H}c_T\varphi^{loc}_{T,\ell,\varepsilon}$, we obtain that $e\coloneqq \PiGalSol-\PiSLODtest\in V$ fulfills:
	\begin{align*}
	\seminorm{e}{V}^2
	&=\frac{1}{\varepsilon}a(\PiGalSol-\PiSLODtest,e)
	=\frac{1}{\varepsilon}\sum_{T\in\mcT_H}c_T a(\varphi_{T,\ell,\varepsilon}-\varphi_{T,\ell,\varepsilon}^{loc},e)\\
	&= - \sum_{T\in\mcT_H}c_T \dual{\gamma_{\partial_n} \varphi_{T,\ell,\varepsilon}^{\rm loc}}{\gamma_0 e}{X^\prime\times X}
	\leq \sum_{T\in\mcT_H} \abs{c_T} \norm{\gamma_{\partial_n} \varphi_{T,\ell,\varepsilon}^{\rm loc}}{L^2(\partial \omega)} C_{\gamma_0} \norm{e}{H^1(\omega)}\\
	&\leq \sigma(\varepsilon,H,\ell) \sum_{T\in\mcT_H} \abs{c_T} C_{\gamma_0} \norm{e}{H^1(\omega)},
	\end{align*}
	where we employed \cref{lem:SLOD_loc_error} in the third equality and~\cref{eq:sigma} in the last inequality. For simplicity, we omit the dependence of $\sigma$ and $C_{rb}$ on $\eps,H$ and $\ell$ in the rest of the proof. As a consequence, thanks to \cref{ass:g_basis}, \cref{eq:f_coll}, the Poincar\'e inequality and~\cref{eq:projHa}, there holds:
	\begin{align*}
		\seminorm{e}{V}^2
		&\lesssim  \sigma \sum_{T\in\mcT_H} \abs{c_T} \norm{e}{H^1(\omega)}
		\lesssim \sigma \sqrt{\sum_{T\in\mcT_H} \abs{c_T}^2} 
		\sqrt{\sum_{T\in\mcT_H}\norm{e}{H^1(\omega)}^2}\\
		&\lesssim \sigma \left(C_{rb}^{1/2}\norm{\proj_Hf}{L^2(\Omega)}\right) \left(C_{ol}\ell^{d/2}\norm{e}{H^1(\Omega)}\right)
		\lesssim \sigma \left(C_{rb}^{1/2}\norm{f}{L^2(\Omega)}\right)
		\left(C_{ol}\ell^{d/2}\right)\seminorm{e}{V}, 	
	\end{align*}
	where $C_{ol}^2\ell^d$ bounds the number of patches containing a fixed mesh element. 
	In particular, we have proved that
	\[
	\seminorm{e}{V}\lesssim\sigma C_{rb}^{1/2} \ell^{d/2}\norm{f}{L^2(\Omega)},
	\]
	so that the estimate \cref{eq:SLOD_conv} follows.
\end{proof}

As previously observed for the ideal multi-scale method, upper bounds on the SLOD error could be derived in the abstract setting $\mcA^{-1}\colon X\rightarrow Y$, for suitable Banach spaces $X$ and $Y$.

Let us point out that, in the case of a piecewise constant right-hand side $f$, the first term in \cref{eq:SLOD_conv} vanishes. Moreover, the $\varepsilon$-dependence of the second expression is dominated by the exponentially decaying quantity $\sigma(\eps,H,\ell)$. Making use of \cref{con:sup_exp_decay}, we derive that the oversampling condition $\ell \gtrsim \abs{\log (\varepsilon H)}^{\frac{d-1}{d}}$ guarantees that the SLOD error convergences with order $H$.

\section{Numerical implementation and stable selection of basis}
\label{sec:numerical_implementation}

This section discusses the implementation of the proposed numerical method, with particular attention to the computation of a basis $\{\varphi_{T,\ell,\varepsilon}^{\rm loc}\,|\, T\in\mcT_H\}$ for the ansatz space $\SLODSpace$ which is associated with a  basis $\{g_{T,\ell,\varepsilon}^{\rm loc}\,|\, T\in\mathcal T_H\}$ of $\mcP^0(\mcT_H)$ via ~\cref{eq:phi_loc}. The Riesz stability of the basis in the sense of \cref{ass:g_basis} has to be respected. 

For simplicity, we take $\Omega$ as the unit hypercube in $d$ dimensions, i.e., $\Omega=(0,1)^d$, discretized by means of a quadrilateral mesh $\mcT_H$.
Given $\ell \geq 1$, we choose an element $T\in\mcT_H$ and consider the corresponding patch $\omega = N^\ell(T)$. In a first step, for each element $K\in\mcT_{H,\omega}$ in the patch mesh, we compute the response of the solution operator restricted to the patch, denoted by $\mcA_{\omega}^{-1}$, to its characteristic function $\mathbf{1}_{K}$, i.e., $\mcA_{\omega}^{-1}\mathbf{1}_{K}$.
By construction, the target basis function $\varphi_{T,\ell,\varepsilon}^{\rm loc}$ is in the span of these $\#\mcT_{H,\omega}\approx \ell^d$ local responses. In a second step, we search the function $\varphi^{\rm loc}_{T,\ell,\varepsilon} \in \operatorname{span}\left\lbrace \mcA_{\omega}^{-1}\mathbf{1}_{K} | K\in \mcT_{H,\omega} \right\rbrace$ in this low-dimensional space by minimizing normal derivatives subject to a unit mass constraint. This constraint minimization is realized by computing the smallest eigenvalue of the symmetric positive (semi-)definite matrix 
\begin{equation}
\left(\frac{1}{|T||K|}\int_{\partial\omega}\gamma_{\partial_n}(\mcA_{\omega}^{-1}\mathbf{1}_{K})
\gamma_{\partial_n}(\mcA_{\omega}^{-1}\mathbf{1}_{T})\,\D s\right)_{K,T\in\mcT_{H,\omega}}.
\label{eq:normal_derivative_eigenvalue}
\end{equation}
The corresponding eigenvector $(c_K)_{K\in\mcT_{H,\omega}}$ contains the coefficients of the expansion of $\varphi_{T,\ell,\varepsilon}^{\rm loc}$ in terms of the local responses. At the same time, the coefficients are the values of $g_{T,\ell,\varepsilon}^{\rm loc}$ in the elements of the patch. 

Unfortunately, the smallest eigenvalue may not be simple or there might be a cluster of small eigenvalues. Then a particular choice of eigenfunction may not always be favorable with regard to the global stability of the basis in the sense of \cref{ass:g_basis}. Especially for patches that touch the boundary of the global domain $\Omega$, an additional optimization step ensures a linear independence of the functions computed in different patches. For this purpose, we incorporate eigenfunctions associated with a certain range of the lowermost eigenvalues. Given all the eigenvalues $\lambda_{1}\leq \lambda_2\leq\dots\leq\lambda_{\#\mcT_{H,\omega}}$ and some parameter $p\geq1$, we choose all indices $1\leq i \leq \#\mcT_{H,\omega}$ so that
\begin{align}
\label{eq:lambda_condition}
\frac{\lambda_{i}}{\lambda_{\#\mcT_{H,\omega}}} \leq \max\left\lbrace \left(\frac{\lambda_{1}}{\lambda_{\#\mcT_{H,\omega}}}\right)^{\frac{1}{p}} , 1e{-}10 \right\rbrace,
\end{align}
and we denote the resulting set of indices by $I$. The choice $p = 1$ reflects the case where only the smallest (potentially multiple) eigenvalue is used, and thus we use $p>1$ in our implementation.

Among these candidate functions with close to minimal normal derivative at the boundary of the patch, we choose the one that maximizes a weighted $L^2(\omega)$-norm under the unit mass constraint. The piecewise constant weight function is zero in the central element $T$ and grows in a $b$-dependent way with a certain distance from the central element.  Let us introduce the midpoints $m_{T},m_{K}\in\R^d$ of the central element and an element of the patch, respectively. We define the distance $\operatorname{dist}(T,K)$ between the elements as
\begin{align*}
\operatorname{dist}(T,K) \coloneqq H (m_{K} - m_{T}) \in \mathbb{Z}^{d},
\end{align*}
and introduce for each element $K\in \mcT_{H,\omega}\setminus T$ its weight by
\begin{align}
\label{eq:weight}
w_{K} \coloneqq \norm{\operatorname{dist}(T,K) - \frac{b(m_{T})}{\norm{b(m_{T})}{2}}}{\ell^\infty}^{p_{w}},
\end{align}
with $p_{w}\geq1$ parameter.
By \cref{eq:weight} we ensure that the elements in the direction of $b$ are less penalized. For a realization for an order $1$ patch, see \cref{fig:weights}.
\begin{figure}
	\centering
	\scalebox{.7}{
%
%
\begin{tikzpicture}

\begin{axis}[%
width=3.566in,
height=3.566in,
at={(1.236in,0.481in)},
scale only axis,
xmin=-0.1,
xmax=0.2,
xticklabels={},
ymin=-0.1,
ymax=0.2,
yticklabels={},
axis x line=none,
axis y line=none
]

\addplot[area legend, line width=1.0pt, table/row sep=crcr, patch, mesh, patch type=rectangle, color=black, forget plot, patch table with point meta={%
0	1	2	3	0\\
4	5	6	7	0\\
8	9	10	11	0\\
12	13	14	15	0\\
16	17	18	19	0\\
20	21	22	23	0\\
24	25	26	27	0\\
28	29	30	31	0\\
32	33	34	35	0\\
}]
table[row sep=crcr] {%
x	y\\
0	0\\
0.0625	0\\
0.0625	0.0625\\
0	0.0625\\
0.0625	0\\
0.125	0\\
0.125	0.0625\\
0.0625	0.0625\\
0.125	0\\
0.1875	0\\
0.1875	0.0625\\
0.125	0.0625\\
0	0.0625\\
0.0625	0.0625\\
0.0625	0.125\\
0	0.125\\
0.0625	0.0625\\
0.125	0.0625\\
0.125	0.125\\
0.0625	0.125\\
0.125	0.0625\\
0.1875	0.0625\\
0.1875	0.125\\
0.125	0.125\\
0	0.125\\
0.0625	0.125\\
0.0625	0.1875\\
0	0.1875\\
0.0625	0.125\\
0.125	0.125\\
0.125	0.1875\\
0.0625	0.1875\\
0.125	0.125\\
0.1875	0.125\\
0.1875	0.1875\\
0.125	0.1875\\
};

\node[align=left]
at (0.03125,0.03125) {9.7012};
\node[align=left]
at (0.09375,0.03125) {7.3087};
\node[align=left]
at (0.15625,	0.03125) {7.3087};
\node[align=left]
at (0.03125,	0.09375) {9.7012};
\node[align=left]
at (0.09375,	0.09375) {0};
\node[align=left]
at (0.15625,	0.09375) {0.1722};
\node[align=left]
at (0.03125,	0.15625) {9.7012};
\node[align=left]
at (0.09375,	0.15625) {0.3422};
\node[align=left]
at (0.15625,	0.15625) {0.0160};
\end{axis}
\end{tikzpicture}%
	}
	\caption{Weights $w_{K}$ for an element that does not reach the boundary, for constant velocity $b$ as given in \cref{eq:b_cos_sin}, and $p_{w} = 2$.}
	\label{fig:weights}
\end{figure}
Eventually, we search the function in the space of the previously selected candidate right-hand sides $\operatorname{span}\left\lbrace g_{T,\ell,\varepsilon,i} | i\in I\right\rbrace$ that minimizes a weighted $L^2(\omega)$-norm subject to the unit mass constraint. This constraint minimization is realized by computing the smallest eigenvalue of the symmetric positive definite matrix 
\begin{align}
\left(\frac{1}{\norm{g_{T,\ell,\varepsilon,i}}{L^2(\omega)}\norm{g_{T,\ell,\varepsilon,j}}{L^2(\omega)}}\sum_{K\in \mcT_{H,\omega}} \int_{K} w_{K} g_{T,\ell,\varepsilon,i}g_{T,\ell,\varepsilon,j} \D x\right)_{i,j\in I}. 
\label{eq:L2w-norm}
\end{align}

In this way, we compute for every element $T$ of the coarse mesh $\mcT_H$ the basis function $\varphi^{\rm loc}_{T,\ell,\varepsilon}$ and hence build the space $V_{H,\ell}$. From our numerical experiments, the choices $p = 1.5$ and $ p_{w}=2 $ produce good results. In \cref{alg:version2} we detail the full algorithm for the computation of the super-localized basis.

\begin{algorithm}
	\caption{Basis selection}
	\label{alg:version2}
	\begin{algorithmic}[1]
		\Require $\ell\geq 1$, $p> 1$, $p_w\geq 1$
		\Ensure stable bases $\{g_{T,\ell,\varepsilon}^{\rm loc}\,|\, T\in\mcT_H\}$ for $\mcP^0(\mcT_H)$
		\For{$T\in\mcT_H$}
		\State set $\omega=N^\ell(T)$, $\Psi_T=\emptyset$, $\Lambda_T=\emptyset$, $\Xi_T=\emptyset$ and $G_T=\emptyset$
		\For{$K\in\mcT_{H,\omega}$}	
		\State $\Psi_T\gets\mcA_{\omega}^{-1}\mathbf{1}_{K}$
		\State compute the weight $w_K$
		\Comment See \cref{eq:weight}
		\EndFor
		\State $\Lambda=(\lambda_1,\ldots,\lambda_{\#\mcT_{H,\omega}}), \Xi = (\xi_1,\ldots,\xi_{\#\mcT_{H,\omega}})\gets \Psi_T$
		\Comment{Eigendecomposition}
		\State $\Lambda_T\gets \lambda_i$ fulfilling~\cref{eq:lambda_condition}
		\State $\Xi_T\gets \xi_i$ s.t. $\lambda_i\in\Lambda_T$
		\State $G_T\gets g_{T,\ell,\varepsilon,i}$ with $g_{T,\ell,\varepsilon,i} = \mcA\xi_i$ for $\xi_i\in\Xi_T$
		\State $ g_{T,\ell,\varepsilon}^{\rm loc} \gets g_{T,\ell,\varepsilon,i} $ eigenvector corresponding to smallest eigenvalue of \cref{eq:L2w-norm}
		\State $\varphi_{T,\ell,\varepsilon}^{\rm loc} \gets g_{T,\ell,\varepsilon}^{\rm loc}$
		\Comment See \cref{eq:phi_loc}
		\EndFor
	\end{algorithmic}
\end{algorithm}

\section{Numerical experiments}
\label{sec:numerical_experiments}

In this section, we demonstrate the performance of our method. For that purpose, we briefly introduce the general configuration. All our experiments were performed in Matlab. The computational domain $\Omega$ is given as the unit hypercube in $d$ dimensions, i.e., $\Omega=(0,1)^d$. We introduce a fine quadrilateral mesh $\mcT_h$ which is supposed to resolve the small parameter $\eps$ and hence serves for the computation of reference solutions and to compute the localized basis functions by means of the standard Galerkin FE method on the space of piecewise bilinear polynomials. Moreover, we consider a coarse quadrilateral mesh $\mcT_H$ as a target scale that does not resolve $\eps$.

\subsection{Two-dimensional experiment}

We start by presenting two-dimensional experiments and compare our approach with the streamline upwind/Petrov--Galerkin (SUPG) method (see~\cite{Franca-Frey-Hughes}) that we briefly recall below. Let $U_H$ denote the standard Galerkin FE space of piecewise bilinear polynomials on the coarse mesh $\mcT_H$.
The SUPG approximation $u^{SUPG}_H\in U_H$ satisfies, for all $v_H\in U_H$
\[
B_{SUPG}(u_H^{SUPG},v_H) = F_{SUPG} (v_H), 
\]
with 
\[
B_{SUPG}(u_H^{SUPG},v_H)\coloneqq a(u_H^{SUPG},v_H) +
\delta_{SUPG}\sum_{T\in\mcT_H}\inner{b\cdot\nabla u_H^{SUPG}}{b\cdot\nabla v_H}{L^2(T)}
\]
and
\[
F_{SUPG} (v_H) = \dual{f}{v_H}{H^{-1}(\Omega)\times H^{1}_0(\Omega)} +
\delta_{SUPG}\sum_{T\in\mcT_H}\inner{f}{b\cdot\nabla v_H}{L^2(T)}.
\]
The symbol $\delta_{SUPG}$ denotes the stabilization parameter, and in the numerics it is chosen as $ \delta_{SUPG} = H^2 $.

\subsubsection{Constant velocity}

First, we follow the experiment from \cite[Section 6]{Li-Peterseim-Schedensack}. We choose the right-hand side $f\equiv 1$ and the constant velocity field $b$ as
\begin{align}
b(x) = \begin{pmatrix}
\cos(0.7) & \sin(0.7)
\end{pmatrix}^\top. \label{eq:b_cos_sin}
\end{align}
The singular perturbation parameter $\eps$ is chosen to be $2^{-7}$. In this configuration we expect boundary layers at the right and top boundaries. Moreover, in this situation the right-hand side is piecewise constant and hence, the first expression in our error estimate in \cref{eq:SLOD_conv} vanishes. Therefore, we observe the localization error. 

The mesh size $h$ of the fine mesh $\mcT_h$ is chosen as $h = 2^{-10}$. \Cref{fig:solutions_2d_const} shows the corresponding reference solution, its FE and SLOD approximation on a coarse mesh with $H = 2^{-4}$. The oversampling parameter $\ell$ is chosen equal to 1.
\begin{figure}
	\centering
	\subfigure[]{\includegraphics[width=.3\linewidth]{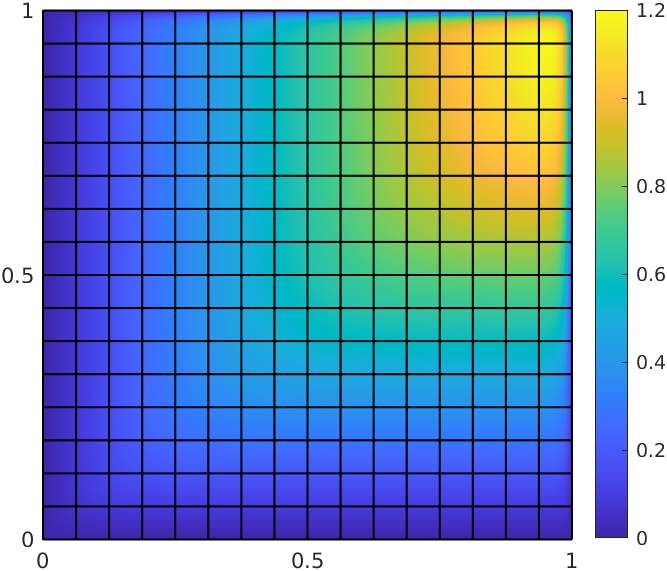}}
	\subfigure[]{\includegraphics[width=.3\linewidth]{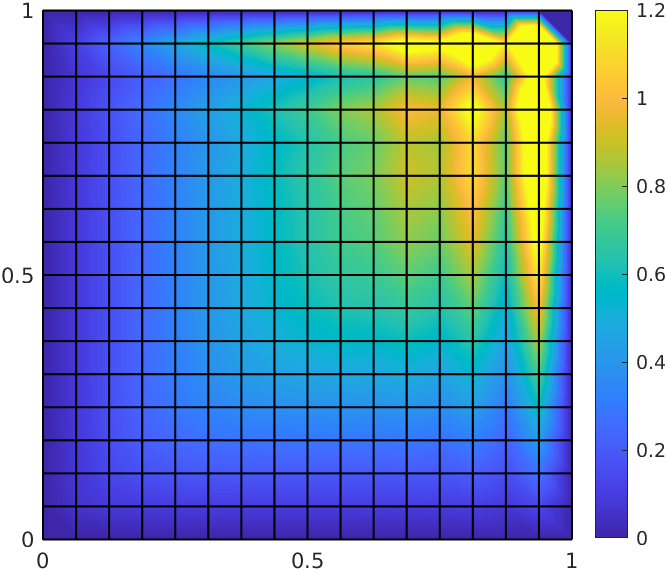}}
	\subfigure[]{\includegraphics[width=.3\linewidth]{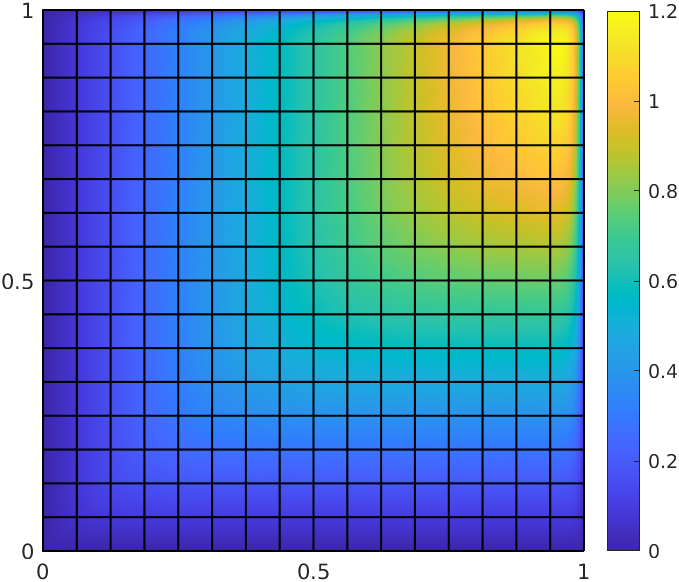}}
	\caption{Reference solution computed on fine mesh with $h=2^{-10}$ (a), FE approximation (b) and SLOD approximation with $\ell = 1$ (c) on coarse mesh with $H=2^{-4}$ for the constant velocity field $b$ given in \cref{eq:b_cos_sin}, right-hand side $f\equiv 1$ and $\eps = 2^{-7}$.}
	\label{fig:solutions_2d_const}
\end{figure}

We observe that the SLOD resolves the layer, whereas the classic FE approximation suffers from severe instabilities. \Cref{fig:convergence_2d_const} shows the convergence rates of the SLOD method for different oversampling parameters $\ell$ and coarse mesh sizes $H$ as well as the error of the FE and SUPG methods. 
\begin{figure}
	\centering
%
%
\definecolor{mycolor1}{rgb}{0.00000,0.44700,0.74100}%
\definecolor{mycolor2}{rgb}{0.85000,0.32500,0.09800}%
\definecolor{mycolor3}{rgb}{0.92900,0.69400,0.12500}%
\definecolor{mycolor4}{rgb}{0.49400,0.18400,0.55600}%
\definecolor{mycolor5}{rgb}{0.46600,0.67400,0.18800}%
\definecolor{mycolor6}{rgb}{0.30100,0.74500,0.93300}%
\begin{tikzpicture}

\begin{axis}[%
width=5.749cm,
height=6cm,
at={(0cm,0cm)},
scale only axis,
xmode=log,
xmin=0.01,
xmax=0.25,
xtick={0.015625,0.03125,0.0625,0.125,0.25},
xticklabels={{$2^{-6}$},{$2^{-5}$},{$2^{-4}$},{$2^{-3}$},{$2^{-2}$}},
xminorticks=true,
xlabel style={font=\color{white!15!black}},
xlabel={$H$},
ymode=log,
ymin=1e-10,
ymax=10,
yminorticks=true,
axis background/.style={fill=white},
axis x line*=bottom,
axis y line*=left,
xmajorgrids,
ymajorgrids,
legend style={at={(1.2,-0.4)}, anchor=south, legend columns=5, legend cell align=left, align=left, draw=white!15!black}
]
\addplot [color=mycolor1, line width=1.5pt, mark size=1.9pt, mark=o, mark options={solid, mycolor1}]
  table[row sep=crcr]{%
0.25	0.00858770024728118\\
0.125	0.00476620235964485\\
0.0625	0.00220045173795469\\
0.03125	0.00226384936811157\\
0.015625	0.0102853014839651\\
};
\addlegendentry{$\ell =1$}

\addplot [color=mycolor2, line width=1.5pt, mark size=3.8pt, mark=o, mark options={solid, mycolor2}]
  table[row sep=crcr]{%
0.25	0.0171677411535904\\
0.125	0.24898378016213\\
0.0625	0.00131659270887952\\
0.03125	0.00131208619831275\\
0.015625	0.00016358020261887\\
};
\addlegendentry{$\ell =2$}

\addplot [color=mycolor3, line width=1.5pt, mark size=5.6pt, mark=o, mark options={solid, mycolor3}]
  table[row sep=crcr]{%
0.125	3.20541237088633e-06\\
0.0625	1.25387038233941e-07\\
0.03125	7.9162671920728e-08\\
0.015625	1.27421520525237e-07\\
};
\addlegendentry{$\ell =3$}

\addplot [color=mycolor4, line width=1.5pt, mark size=7.5pt, mark=o, mark options={solid, mycolor4}]
  table[row sep=crcr]{%
0.125	4.23517136735739e-08\\
0.0625	2.31249560556067e-09\\
0.03125	9.87783814781757e-09\\
0.015625	1.42337840846482e-08\\
};
\addlegendentry{$\ell =4$}

\addplot [color=mycolor5, line width=1.5pt, mark size=7.5pt, mark=x, mark options={solid, mycolor5}]
  table[row sep=crcr]{%
0.25	0.780547709427826\\
0.125	0.246602411157313\\
0.0625	0.107345613184665\\
0.03125	0.0383353824108605\\
0.015625	0.0111121402695505\\
};
\addlegendentry{FEM}

\addplot [color=mycolor6, line width=1.5pt, mark size=7.5pt, mark=diamond, mark options={solid, mycolor6}]
  table[row sep=crcr]{%
0.25	0.256671434180938\\
0.125	0.176007311353391\\
0.0625	0.0955051055850821\\
0.03125	0.0370498517206095\\
0.015625	0.011030881161918\\
};
\addlegendentry{SUPG}

\addplot [color=black, dashed, line width=1.5pt]
  table[row sep=crcr]{%
0.25	0.25\\
0.015625	0.015625\\
};
\addlegendentry{$\mathcal{O}(H)$}

\addplot [color=black, dashdotted, line width=1.5pt]
  table[row sep=crcr]{%
0.25	0.0625\\
0.015625	0.000244140625\\
};
\addlegendentry{$\mathcal{O}(H^2)$}

\addplot [color=black, dotted, line width=1.5pt]
  table[row sep=crcr]{%
0.25	0.015625\\
0.015625	3.814697265625e-06\\
};
\addlegendentry{$\mathcal{O}(H^3)$}

\end{axis}

\begin{axis}[%
width=5.749cm,
height=6cm,
at={(7.564cm,0cm)},
scale only axis,
xmode=log,
xmin=0.01,
xmax=0.25,
xtick={0.015625,0.03125,0.0625,0.125,0.25},
xticklabels={{$2^{-6}$},{$2^{-5}$},{$2^{-4}$},{$2^{-3}$},{$2^{-2}$}},
xminorticks=true,
xlabel style={font=\color{white!15!black}},
xlabel={$H$},
ymode=log,
ymin=1e-07,
ymax=100,
yminorticks=true,
axis background/.style={fill=white},
axis x line*=bottom,
axis y line*=left,
xmajorgrids,
ymajorgrids
]
\addplot [color=mycolor1, line width=1.5pt, mark size=1.9pt, mark=o, mark options={solid, mycolor1}, forget plot]
  table[row sep=crcr]{%
0.25	0.324081074447027\\
0.125	0.345055055189613\\
0.0625	0.357112106850395\\
0.03125	0.514725776700479\\
0.015625	1.2613277825365\\
};
\addplot [color=mycolor2, line width=1.5pt, mark size=3.8pt, mark=o, mark options={solid, mycolor2}, forget plot]
  table[row sep=crcr]{%
0.25	0.491070088950502\\
0.125	5.47819925891153\\
0.0625	0.112625271611721\\
0.03125	0.0722974573351938\\
0.015625	0.0215457111618948\\
};
\addplot [color=mycolor3, line width=1.5pt, mark size=5.6pt, mark=o, mark options={solid, mycolor3}, forget plot]
  table[row sep=crcr]{%
0.125	0.000138358579716926\\
0.0625	1.54086681837892e-05\\
0.03125	1.59956882093285e-05\\
0.015625	4.03805446860225e-05\\
};
\addplot [color=mycolor4, line width=1.5pt, mark size=7.5pt, mark=o, mark options={solid, mycolor4}, forget plot]
  table[row sep=crcr]{%
0.125	4.22543576574089e-06\\
0.0625	2.88424088050177e-07\\
0.03125	1.80173660522889e-06\\
0.015625	3.99421451719062e-06\\
};
\addplot [color=mycolor5, line width=1.5pt, mark size=7.5pt, mark=x, mark options={solid, mycolor5}, forget plot]
  table[row sep=crcr]{%
0.25	11.5810606803323\\
0.125	8.92821926917347\\
0.0625	7.37568163036566\\
0.03125	5.1290547865098\\
0.015625	2.95007891323938\\
};
\addplot [color=mycolor6, line width=1.5pt, mark size=7.5pt, mark=diamond, mark options={solid, mycolor6}, forget plot]
  table[row sep=crcr]{%
0.25	7.31292880618133\\
0.125	7.38393318996578\\
0.0625	6.81395680373791\\
0.03125	4.9951256946915\\
0.015625	2.92799846089822\\
};
\addplot [color=black, dashed, line width=1.5pt, forget plot]
  table[row sep=crcr]{%
0.25	0.25\\
0.015625	0.015625\\
};
\addplot [color=black, dashdotted, line width=1.5pt, forget plot]
  table[row sep=crcr]{%
0.25	0.0625\\
0.015625	0.000244140625\\
};
\end{axis}
\end{tikzpicture}%
	\caption{Error in $L^2$- (left) and $\seminorm{\bullet}{V}$-norm (right) for the constant velocity field $b$ as in \cref{eq:b_cos_sin} and right-hand side $f\equiv 1$ with $\eps = 2^{-7}$.}
	\label{fig:convergence_2d_const}
\end{figure}
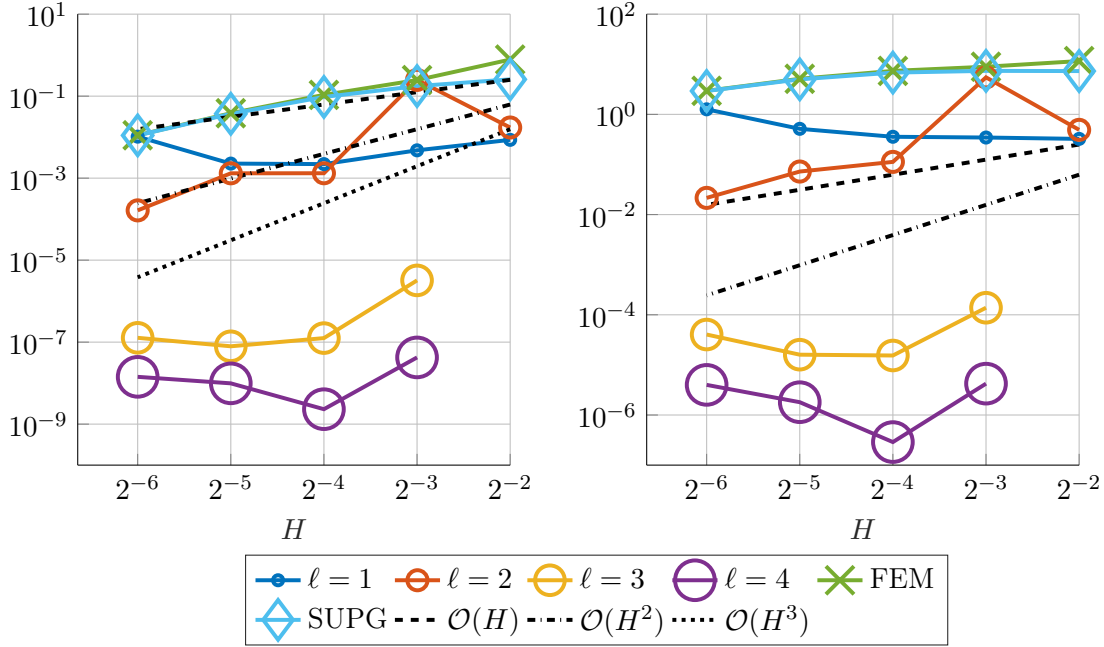

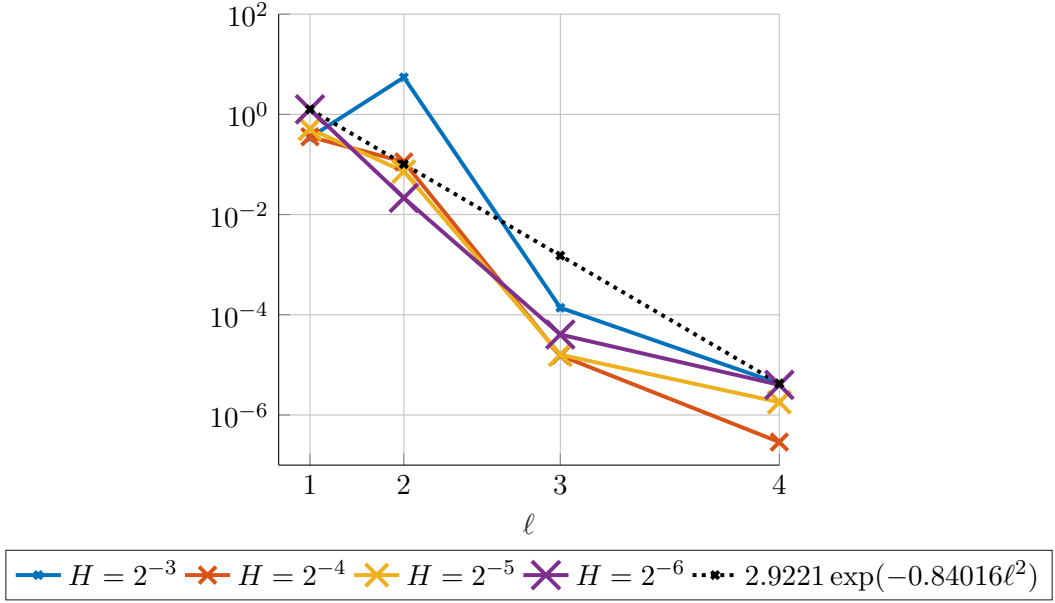
\begin{figure}
	\centering
%
%
\definecolor{mycolor1}{rgb}{0.00000,0.44700,0.74100}%
\definecolor{mycolor2}{rgb}{0.85000,0.32500,0.09800}%
\definecolor{mycolor3}{rgb}{0.92900,0.69400,0.12500}%
\definecolor{mycolor4}{rgb}{0.49400,0.18400,0.55600}%
\begin{tikzpicture}

\begin{axis}[%
width=6.656cm,
height=6cm,
at={(0cm,0cm)},
scale only axis,
xmin=0,
xmax=16,
xtick={1,4,9,16},
xticklabels={{1},{2},{3},{4}},
xlabel style={font=\color{white!15!black}},
xlabel={$\ell$},
ymode=log,
ymin=1e-07,
ymax=100,
yminorticks=true,
ylabel style={font=\color{white!15!black}},
axis x line*=bottom,
axis y line*=left,
xmajorgrids,
ymajorgrids,
axis background/.style={fill=white},
legend style={at={(0.5,-0.3)}, anchor=south, legend columns=5, legend cell align=left, align=left, draw=white!15!black}
]

\addplot [color=mycolor1, line width=1.5pt, mark=x, mark options={solid, mycolor1}]
  table[row sep=crcr]{%
1	0.345055055189612\\
4	5.47819925891152\\
9	0.000138358579716926\\
16	4.2254357657409e-06\\
};
\addlegendentry{$H = 2^{-3}$}

\addplot [color=mycolor2, line width=1.5pt, mark size=4.5pt, mark=x, mark options={solid, mycolor2}]
  table[row sep=crcr]{%
1	0.357112106850397\\
4	0.112625271611721\\
9	1.54086681837892e-05\\
16	2.88424088050178e-07\\
};
\addlegendentry{$H = 2^{-4}$}

\addplot [color=mycolor3, line width=1.5pt, mark size=6.0pt, mark=x, mark options={solid, mycolor3}]
  table[row sep=crcr]{%
1	0.514725776700478\\
4	0.0722974573351941\\
9	1.59956882093285e-05\\
16	1.80173660522889e-06\\
};
\addlegendentry{$H = 2^{-5}$}

\addplot [color=mycolor4, line width=1.5pt, mark size=7.5pt, mark=x, mark options={solid, mycolor4}]
  table[row sep=crcr]{%
1	1.2613277825365\\
4	0.0215457111618948\\
9	4.03805446860226e-05\\
16	3.9942145171906e-06\\
};
\addlegendentry{$H = 2^{-6}$}

\addplot [color=black, dotted, line width=1.5pt, mark=x, mark options={solid, black}]
  table[row sep=crcr]{%
1	1.2613036705845\\
4	0.101434101726452\\
9	0.00151981665986239\\
16	4.24268999039379e-06\\
};
\addlegendentry{$2.9221\exp(-0.84016\ell^{2})$}

\end{axis}
\end{tikzpicture}%
	\caption{Decay of the localization error in $\seminorm{\bullet}{V}$-norm versus the oversampling parameter $\ell$ for $\eps = 2^{-7}$ and various values of $H$.}	
	\label{fig:sup_exp_decay_2d_const}
\end{figure}

The super-exponential convergence of \cref{con:sup_exp_decay} is numerically verified in \cref{fig:sup_exp_decay_2d_const}. Unfortunately, our method shows inaccuracies for refinements in $H$. Most likely, these are due to the selection of the basis functions as discussed in \cref{sec:numerical_implementation} and hence an improvement in this selection process could lead to a more accurate method. However, we chose to ensure stability and possibly lose some accuracy in return.

\subsubsection{Variable velocity and non-constant right-hand side}

In this example, we consider a varying velocity field $b$, given as
\begin{align}
b(x) = \begin{pmatrix}
-x_{2} & x_{1}
\end{pmatrix}^\top,\quad \text{for all }x\in\Omega. \label{eq:b_circle}
\end{align}
Hence, the penalization introduced in \cref{sec:numerical_implementation} varies with the different macroscopic cells in the coarse mesh. The flow in this example yields a boundary layer at the left boundary. We choose the non-constant right-hand side as $f(x) = \sin(\pi x_{1}) \cos(\pi x_{2})$. Thus, with respect to the error estimate in \cref{thm:SLOD_conv} the first expression in \cref{eq:SLOD_conv} does not vanish, and we expect a convergence of order one in the $\seminorm{\bullet}{V}$-norm as the right-hand side is regular enough. \Cref{fig:solutions_2d_variable} shows the reference solution as well as its FE and SLOD approximations. We again observe that the SLOD resolves the boundary layer, whereas the FEM delivers poor results.
\begin{figure}
	\centering
	\includegraphics[width=.3\linewidth]{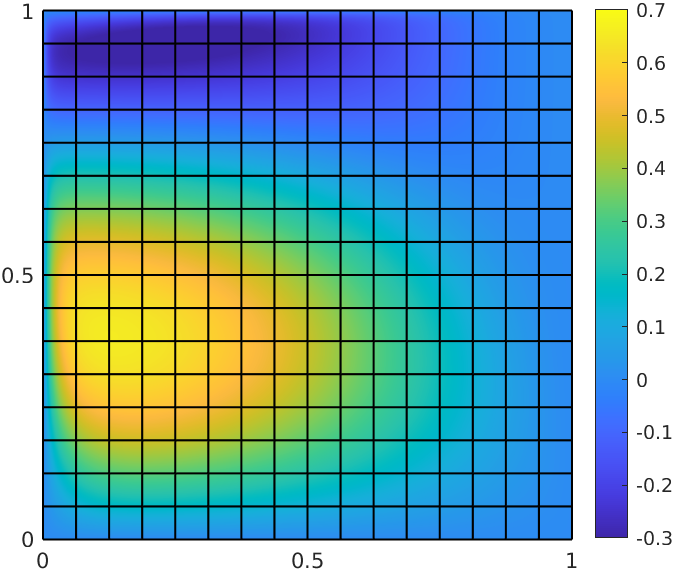}
	\includegraphics[width=.3\linewidth]{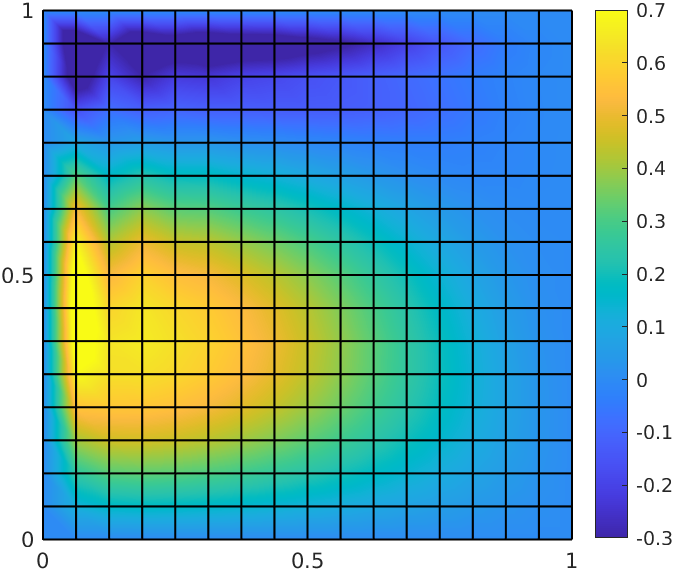}
	\includegraphics[width=.3\linewidth]{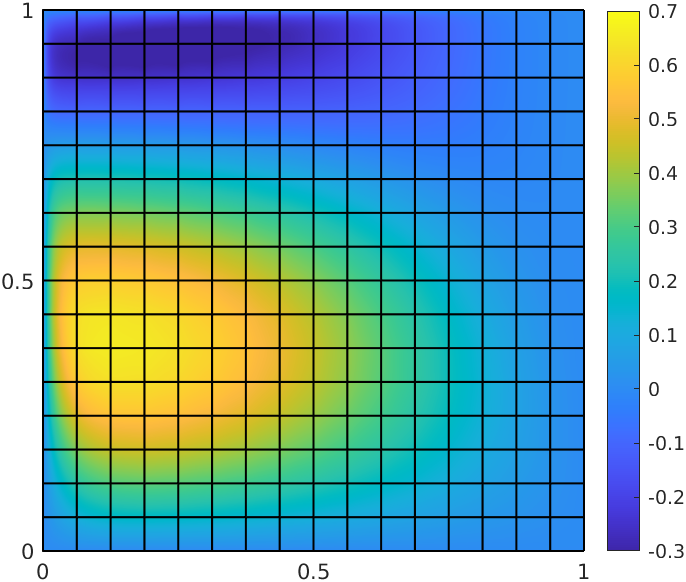}
	\caption{Reference solution computed on fine mesh with $h=2^{-10}$ (a), FE approximation (b) and SLOD approximation with $\ell= 1$ (c) on coarse mesh with $H=2^{-4}$ for the variable velocity field $b$ given in \cref{eq:b_circle}, right-hand side $f= \sin(\pi x_{1}) \cos(\pi x_{2})$ and $\eps = 2^{-7}$.}
	\label{fig:solutions_2d_variable}
\end{figure}

For non-constants right-hand sides, we expect improved approximation properties of the SLOD method~\cref{eq:VHell} with $\Pi_H f$ replaced by $f$. We will be refer to such method as SLOD-Galerkin. 
Note that the two methods produce different approximations and require different computational efforts. More in details, since both methods look for an approximation in the same ansatz space $\mcV_H^\ell$, they share the offline phase, namely, the computation of the set of operator-adapted local basis functions. On the other hand, they differ in the online phase, when the actual approximation to the solution is computed. In particular, the SLOD method is more online efficient than the SLOD-Galerkin. 

The convergence in the $L^2(\Omega)$- and $\seminorm{\bullet}{V}$-norm for both the SLOD and the SLOD-Galerkin are shown in \cref{fig:convergence_2d_variable}, where we observe second order convergence in the $\seminorm{\bullet}{V}$-norm in $H$ for both methods. In the $L^2(\Omega)$-norm, the SLOD-Galerkin has a convergence of order three, one order better than our proposed method, which is due to the additional information. 

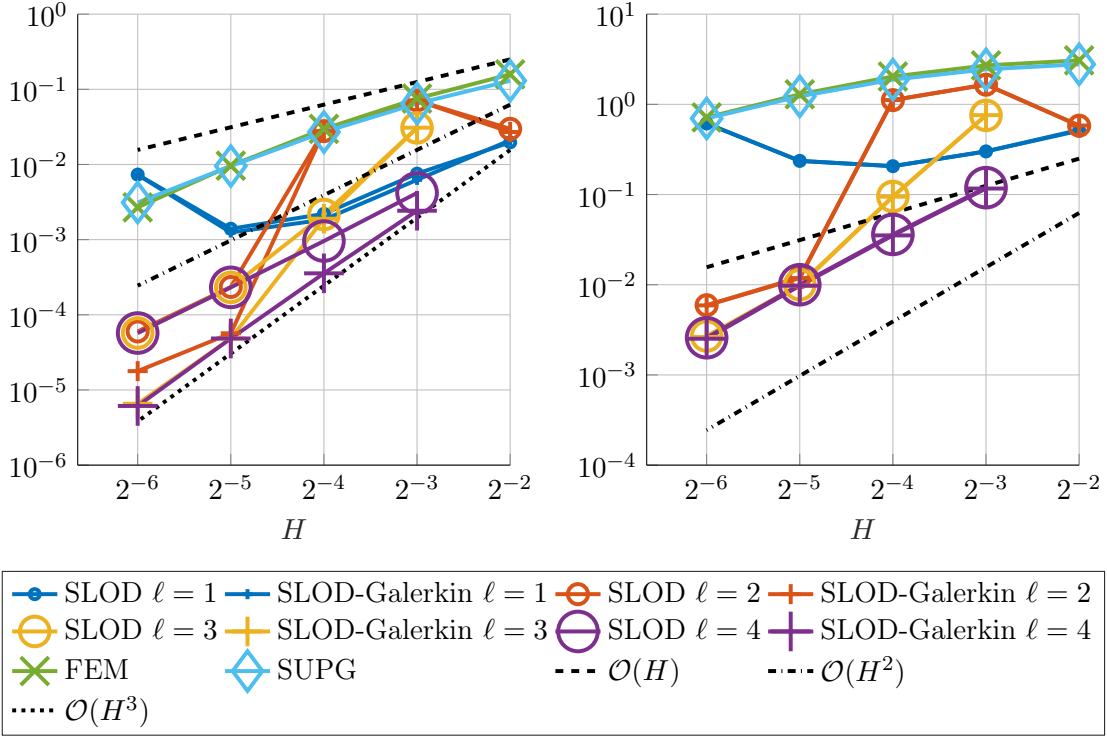
\begin{figure}
	\centering
%
%
\definecolor{mycolor1}{rgb}{0.00000,0.44700,0.74100}%
\definecolor{mycolor2}{rgb}{0.85000,0.32500,0.09800}%
\definecolor{mycolor3}{rgb}{0.92900,0.69400,0.12500}%
\definecolor{mycolor4}{rgb}{0.49400,0.18400,0.55600}%
\definecolor{mycolor5}{rgb}{0.46600,0.67400,0.18800}%
\definecolor{mycolor6}{rgb}{0.30100,0.74500,0.93300}%
\begin{tikzpicture}

\begin{axis}[%
width=5.749cm,
height=6cm,
at={(0cm,0cm)},
scale only axis,
xmode=log,
xmin=0.01,
xmax=0.25,
xtick={0.015625,0.03125,0.0625,0.125,0.25},
xticklabels={{$2^{-6}$},{$2^{-5}$},{$2^{-4}$},{$2^{-3}$},{$2^{-2}$}},
xminorticks=false,
xlabel style={font=\color{white!15!black}},
xlabel={$H$},
ymode=log,
ymin=1e-06,
ymax=1,
yminorticks=false,
axis background/.style={fill=white},
axis x line*=bottom,
axis y line*=left,
xmajorgrids,
ymajorgrids,
legend style={at={(1.1,-0.6)}, anchor=south, legend columns=4, legend cell align=left, align=left, draw=white!15!black}
]
\addplot [color=mycolor1, line width=1.5pt, mark size=1.9pt, mark=o, mark options={solid, mycolor1}]
  table[row sep=crcr]{%
0.25	0.0197153894210973\\
0.125	0.00754902080331059\\
0.0625	0.00220324266360721\\
0.03125	0.00139573561573549\\
0.015625	0.0073309295499131\\
};
\addlegendentry{SLOD  $\ell =1$}

\addplot [color=mycolor1, line width=1.5pt, mark size=1.9pt, mark=+, mark options={solid, mycolor1}]
  table[row sep=crcr]{%
0.25	0.0210461356088747\\
0.125	0.00623721334498726\\
0.0625	0.00184546030389234\\
0.03125	0.00124470713738967\\
0.015625	0.00728330509048191\\
};
\addlegendentry{SLOD-Galerkin  $\ell =1$}

\addplot [color=mycolor2, line width=1.5pt, mark size=3.8pt, mark=o, mark options={solid, mycolor2}]
  table[row sep=crcr]{%
0.25	0.0297920791635263\\
0.125	0.0710172627006931\\
0.0625	0.0277959132388842\\
0.03125	0.00023454229946089\\
0.015625	6.00346932172156e-05\\
};
\addlegendentry{SLOD  $\ell =2$}

\addplot [color=mycolor2, line width=1.5pt, mark size=3.8pt, mark=+, mark options={solid, mycolor2}]
  table[row sep=crcr]{%
0.25	0.0270060057674967\\
0.125	0.0709589907880768\\
0.0625	0.0277467794986627\\
0.03125	5.67208426275005e-05\\
0.015625	1.78029308910442e-05\\
};
\addlegendentry{SLOD-Galerkin  $\ell =2$}

\addplot [color=mycolor3, line width=1.5pt, mark size=5.6pt, mark=o, mark options={solid, mycolor3}]
  table[row sep=crcr]{%
0.125	0.0306651569254923\\
0.0625	0.00215983904491433\\
0.03125	0.000231719219144313\\
0.015625	5.72959164163961e-05\\
};
\addlegendentry{SLOD  $\ell =3$}

\addplot [color=mycolor3, line width=1.5pt, mark size=5.6pt, mark=+, mark options={solid, mycolor3}]
  table[row sep=crcr]{%
0.125	0.030833221617003\\
0.0625	0.00190000499041617\\
0.03125	4.83810112953034e-05\\
0.015625	6.54645860350975e-06\\
};
\addlegendentry{SLOD-Galerkin  $\ell =3$}

\addplot [color=mycolor4, line width=1.5pt, mark size=7.5pt, mark=o, mark options={solid, mycolor4}]
  table[row sep=crcr]{%
0.125	0.00417913710937891\\
0.0625	0.000957781188665228\\
0.03125	0.000231721197487305\\
0.015625	5.72456833943152e-05\\
};
\addlegendentry{SLOD  $\ell =4$}

\addplot [color=mycolor4, line width=1.5pt, mark size=7.5pt, mark=+, mark options={solid, mycolor4}]
  table[row sep=crcr]{%
0.125	0.00242587572732592\\
0.0625	0.000356200145193131\\
0.03125	4.83781479039701e-05\\
0.015625	6.12489491974141e-06\\
};
\addlegendentry{SLOD-Galerkin  $\ell =4$}

\addplot [color=mycolor5, line width=1.5pt, mark size=7.5pt, mark=x, mark options={solid, mycolor5}]
  table[row sep=crcr]{%
0.25	0.158273517857702\\
0.125	0.0753248697115743\\
0.0625	0.0297361880943194\\
0.03125	0.00963745655378103\\
0.015625	0.00268630731663274\\
};
\addlegendentry{FEM}

\addplot [color=mycolor6, line width=1.5pt, mark size=7.5pt, mark=diamond, mark options={solid, mycolor6}]
  table[row sep=crcr]{%
0.25	0.129954510604428\\
0.125	0.0646418502353033\\
0.0625	0.0270249105661935\\
0.03125	0.00952128062639556\\
0.015625	0.00311776413909312\\
};
\addlegendentry{SUPG}

\addplot [color=black, dashed, line width=1.5pt]
  table[row sep=crcr]{%
0.25	0.25\\
0.015625	0.015625\\
};
\addlegendentry{$\mathcal{O}(H)$}

\addplot [color=black, dashdotted, line width=1.5pt]
  table[row sep=crcr]{%
0.25	0.0625\\
0.015625	0.000244140625\\
};
\addlegendentry{$\mathcal{O}(H^2)$}

\addplot [color=black, dotted, line width=1.5pt]
  table[row sep=crcr]{%
0.25	0.015625\\
0.015625	3.814697265625e-06\\
};
\addlegendentry{$\mathcal{O}(H^3)$}

\end{axis}

\begin{axis}[%
width=5.749cm,
height=6cm,
at={(7.564cm,0cm)},
scale only axis,
xmode=log,
xmin=0.01,
xmax=0.25,
xtick={0.015625,0.03125,0.0625,0.125,0.25},
xticklabels={{$2^{-6}$},{$2^{-5}$},{$2^{-4}$},{$2^{-3}$},{$2^{-2}$}},
xminorticks=false,
xlabel style={font=\color{white!15!black}},
xlabel={$H$},
ymode=log,
ymin=0.0001,
ymax=10,
yminorticks=false,
axis background/.style={fill=white},
axis x line*=bottom,
axis y line*=left,
xmajorgrids,
ymajorgrids
]
\addplot [color=mycolor1, line width=1.5pt, mark size=1.9pt, mark=o, mark options={solid, mycolor1}, forget plot]
  table[row sep=crcr]{%
0.25	0.509989985792755\\
0.125	0.300687509455867\\
0.0625	0.206275347315595\\
0.03125	0.236052793377346\\
0.015625	0.613674834156775\\
};
\addplot [color=mycolor1, line width=1.5pt, mark size=1.9pt, mark=+, mark options={solid, mycolor1}, forget plot]
  table[row sep=crcr]{%
0.25	0.520370823647188\\
0.125	0.300974462830501\\
0.0625	0.206450768202272\\
0.03125	0.236131980162322\\
0.015625	0.613731039460794\\
};
\addplot [color=mycolor2, line width=1.5pt, mark size=3.8pt, mark=o, mark options={solid, mycolor2}, forget plot]
  table[row sep=crcr]{%
0.25	0.579890038252885\\
0.125	1.65130517918154\\
0.0625	1.11362041632407\\
0.03125	0.0120627886907905\\
0.015625	0.00592847722543761\\
};
\addplot [color=mycolor2, line width=1.5pt, mark size=3.8pt, mark=+, mark options={solid, mycolor2}, forget plot]
  table[row sep=crcr]{%
0.25	0.582385000410855\\
0.125	1.66182825143697\\
0.0625	1.11522448386114\\
0.03125	0.0118602215559395\\
0.015625	0.00590257398548208\\
};
\addplot [color=mycolor3, line width=1.5pt, mark size=5.6pt, mark=o, mark options={solid, mycolor3}, forget plot]
  table[row sep=crcr]{%
0.125	0.754123285433902\\
0.0625	0.0952234954665771\\
0.03125	0.0100025255056288\\
0.015625	0.00269879757890277\\
};
\addplot [color=mycolor3, line width=1.5pt, mark size=5.6pt, mark=+, mark options={solid, mycolor3}, forget plot]
  table[row sep=crcr]{%
0.125	0.759327468959307\\
0.0625	0.0946675124354339\\
0.03125	0.00975524077748336\\
0.015625	0.00263750116893913\\
};
\addplot [color=mycolor4, line width=1.5pt, mark size=7.5pt, mark=o, mark options={solid, mycolor4}, forget plot]
  table[row sep=crcr]{%
0.125	0.11915644971494\\
0.0625	0.0361126253771928\\
0.03125	0.0100020525257479\\
0.015625	0.00258336807511859\\
};
\addplot [color=mycolor4, line width=1.5pt, mark size=7.5pt, mark=+, mark options={solid, mycolor4}, forget plot]
  table[row sep=crcr]{%
0.125	0.116574546075205\\
0.0625	0.0352260736969477\\
0.03125	0.00975476157514524\\
0.015625	0.00251925256823014\\
};
\addplot [color=mycolor5, line width=1.5pt, mark size=7.5pt, mark=x, mark options={solid, mycolor5}, forget plot]
  table[row sep=crcr]{%
0.25	3.07420329123004\\
0.125	2.70328520350054\\
0.0625	2.03525266364891\\
0.03125	1.29334610354386\\
0.015625	0.716232327345983\\
};
\addplot [color=mycolor6, line width=1.5pt, mark size=7.5pt, mark=diamond, mark options={solid, mycolor6}, forget plot]
  table[row sep=crcr]{%
0.25	2.78018819801733\\
0.125	2.44300170501598\\
0.0625	1.87952834229919\\
0.03125	1.22679562126683\\
0.015625	0.697872248909419\\
};
\addplot [color=black, dashed, line width=1.5pt, forget plot]
  table[row sep=crcr]{%
0.25	0.25\\
0.015625	0.015625\\
};
\addplot [color=black, dashdotted, line width=1.5pt, forget plot]
  table[row sep=crcr]{%
0.25	0.0625\\
0.015625	0.000244140625\\
};
\end{axis}
\end{tikzpicture}%
	\caption{Error in $L^2(\Omega)$- (left) and $\seminorm{\bullet}{V}$-norm (right) for the variable velocity field $b$ as in \cref{eq:b_circle} and right-hand side $f= \sin(\pi x_{1}) \cos(\pi x_{2})$. We consider the parameter $\eps = 2^{-7}$ and the proposed SLOD method as well as the SLOD-Galerkin.}
	\label{fig:convergence_2d_variable}
\end{figure}

\subsection{Three-dimensional experiment}
\label{sec:3D}

As mentioned before, the variational multi-scale stabilization method from \cite{Li-Peterseim-Schedensack} only works in one or two dimensions. Here we show that the super-localized variant is capable to approximate the solution even in a three-dimensional setup. In this configuration we again choose a constant velocity field $b$, which is given as
\begin{align}
\label{eq:b_3d}
b = \frac{\pi}{4}\begin{pmatrix}
1 & 1 & 1
\end{pmatrix}^\top.
\end{align}
The constant right-hand side is $f\equiv 1$. We expect a boundary layer around the top right corner at $\begin{pmatrix}
1 & 1 & 1
\end{pmatrix}^{\top}$. In the three-dimensional setting we compute the reference solution on a mesh with $h = 2^{-6}$, which resolves the chosen $\eps = 2^{-5}$. \Cref{fig:convergence_3d_a} shows the $L^2(\Omega)$- and $\seminorm{\bullet}{V}$-norm errors obtained for the SLOD method. As in our first experiment, due to the constant right-hand side we observe the localization error.
\begin{figure}
	\centering
%
%
\definecolor{mycolor1}{rgb}{0.00000,0.44700,0.74100}%
\definecolor{mycolor2}{rgb}{0.85000,0.32500,0.09800}%
\definecolor{mycolor3}{rgb}{0.92900,0.69400,0.12500}%
\definecolor{mycolor4}{rgb}{0.49400,0.18400,0.55600}%
\definecolor{mycolor5}{rgb}{0.46600,0.67400,0.18800}%
\definecolor{mycolor6}{rgb}{0.30100,0.74500,0.93300}%
\begin{tikzpicture}

\begin{axis}[%
width=5.749cm,
height=6cm,
at={(0cm,0cm)},
scale only axis,
xmode=log,
xmin=0.055,
xmax=0.25,
xtick={0.0625,0.125,0.25},
xticklabels={{$2^{-4}$},{$2^{-3}$},{$2^{-2}$}},
xminorticks=true,
xlabel style={font=\color{white!15!black}},
xlabel={$H$},
ymode=log,
ymin=1e-09,
ymax=1,
yminorticks=true,
axis background/.style={fill=white},
title style={font=\bfseries},
axis x line*=bottom,
axis y line*=left,
xmajorgrids,
ymajorgrids,
legend style={at={(1.2,-0.4)}, anchor=south, legend columns=5, legend cell align=left, align=left, draw=white!15!black}
]
\addplot [color=mycolor1, line width=1.5pt, mark size=1.9pt, mark=o, mark options={solid, mycolor1}]
  table[row sep=crcr]{%
0.25	0.00359428272143508\\
0.125	0.00387724540793881\\
0.0625	0.0236712094013949\\
};
\addlegendentry{$\ell =1$}

\addplot [color=mycolor2, line width=1.5pt, mark size=3.8pt, mark=o, mark options={solid, mycolor2}]
  table[row sep=crcr]{%
0.25	0.0166065873775704\\
0.125	0.00071551572028192\\
0.0625	6.4529421085517e-05\\
};
\addlegendentry{$\ell =2$}

\addplot [color=mycolor3, line width=1.5pt, mark size=5.6pt, mark=o, mark options={solid, mycolor3}]
  table[row sep=crcr]{%
0.125	3.44038987737399e-07\\
0.0625	9.00438772090987e-07\\
};
\addlegendentry{$\ell =3$}

\addplot [color=mycolor4, line width=1.5pt, mark size=7.5pt, mark=o, mark options={solid, mycolor4}]
  table[row sep=crcr]{%
0.125	1.25640612655907e-08\\
0.0625	4.01416577934205e-08\\
};
\addlegendentry{$\ell =4$}

\addplot [color=mycolor5, line width=1.5pt, mark size=7.5pt, mark=x, mark options={solid, mycolor5}]
  table[row sep=crcr]{%
0.25	0.148539095543765\\
0.125	0.0585860890356043\\
0.0625	0.0174972409706877\\
};
\addlegendentry{FEM}

\addplot [color=mycolor6, line width=1.5pt, mark size=7.5pt, mark=diamond, mark options={solid, mycolor6}]
  table[row sep=crcr]{%
	0.25	0.127191184438981\\
	0.125	0.0556936282312964\\
	0.0625	0.017836730768807\\
};
\addlegendentry{SUPG}

\addplot [color=black, dashed, line width=1.5pt]
  table[row sep=crcr]{%
0.25	0.25\\
0.0625	0.0625\\
};
\addlegendentry{$\mathcal{O}(H)$}

\addplot [color=black, dashdotted, line width=1.5pt]
  table[row sep=crcr]{%
0.25	0.0625\\
0.0625	0.00390625\\
};
\addlegendentry{$\mathcal{O}(H^2)$}

\addplot [color=black, dotted, line width=1.5pt]
  table[row sep=crcr]{%
0.25	0.015625\\
0.0625	0.000244140625\\
};
\addlegendentry{$\mathcal{O}(H^3)$}

\end{axis}

\begin{axis}[%
width=5.749cm,
height=6cm,
at={(7.564cm,0cm)},
scale only axis,
xmode=log,
xmin=0.055,
xmax=0.25,
xtick={0.0625,0.125,0.25},
xticklabels={{$2^{-4}$},{$2^{-3}$},{$2^{-2}$}},
xminorticks=true,
xlabel style={font=\color{white!15!black}},
xlabel={$H$},
ymode=log,
ymin=1e-08,
ymax=10,
yminorticks=true,
axis background/.style={fill=white},
title style={font=\bfseries},
axis x line*=bottom,
axis y line*=left,
xmajorgrids,
ymajorgrids
]
\addplot [color=mycolor1, line width=1.5pt, mark size=1.9pt, mark=o, mark options={solid, mycolor1}, forget plot]
  table[row sep=crcr]{%
0.25	0.101220118482965\\
0.125	0.236396769271957\\
0.0625	0.830830305955967\\
};
\addplot [color=mycolor2, line width=1.5pt, mark size=3.8pt, mark=o, mark options={solid, mycolor2}, forget plot]
  table[row sep=crcr]{%
0.25	0.312644056411654\\
0.125	0.0312581453439664\\
0.0625	0.00600399044100064\\
};
\addplot [color=mycolor3, line width=1.5pt, mark size=5.6pt, mark=o, mark options={solid, mycolor3}, forget plot]
  table[row sep=crcr]{%
0.125	1.42287399389329e-05\\
0.0625	0.000100068241620617\\
};
\addplot [color=mycolor4, line width=1.5pt, mark size=7.5pt, mark=o, mark options={solid, mycolor4}, forget plot]
  table[row sep=crcr]{%
0.125	4.83724707952136e-07\\
0.0625	2.71558464531154e-06\\
};
\addplot [color=mycolor5, line width=1.5pt, mark size=7.5pt, mark=x, mark options={solid, mycolor5}, forget plot]
  table[row sep=crcr]{%
0.25	2.76700278486213\\
0.125	2.01341483678796\\
0.0625	1.17866809460072\\
};
\addplot [color=mycolor6, line width=1.5pt, mark size=7.5pt, mark=diamond, mark options={solid, mycolor6}, forget plot]
  table[row sep=crcr]{%
0.25	2.32419515035247\\
0.125	1.84868446855995\\
0.0625	1.14595560389424\\
};

\addplot [color=black, dashed, line width=1.5pt, forget plot]
  table[row sep=crcr]{%
0.25	0.25\\
0.0625	0.0625\\
};
\addplot [color=black, dashdotted, line width=1.5pt, forget plot]
  table[row sep=crcr]{%
0.25	0.0625\\
0.0625	0.00390625\\
};
\end{axis}
\end{tikzpicture}%
	\caption{Error in the $L^2(\Omega)$- (left) and $\seminorm{\bullet}{V}$-norm (right) for the constant velocity field $b$ as in \cref{eq:b_3d}, right-hand side $f\equiv 1$ and $\eps = 2^{-5}$, in the three-dimensional case.}
	\label{fig:convergence_3d_a}
\end{figure}
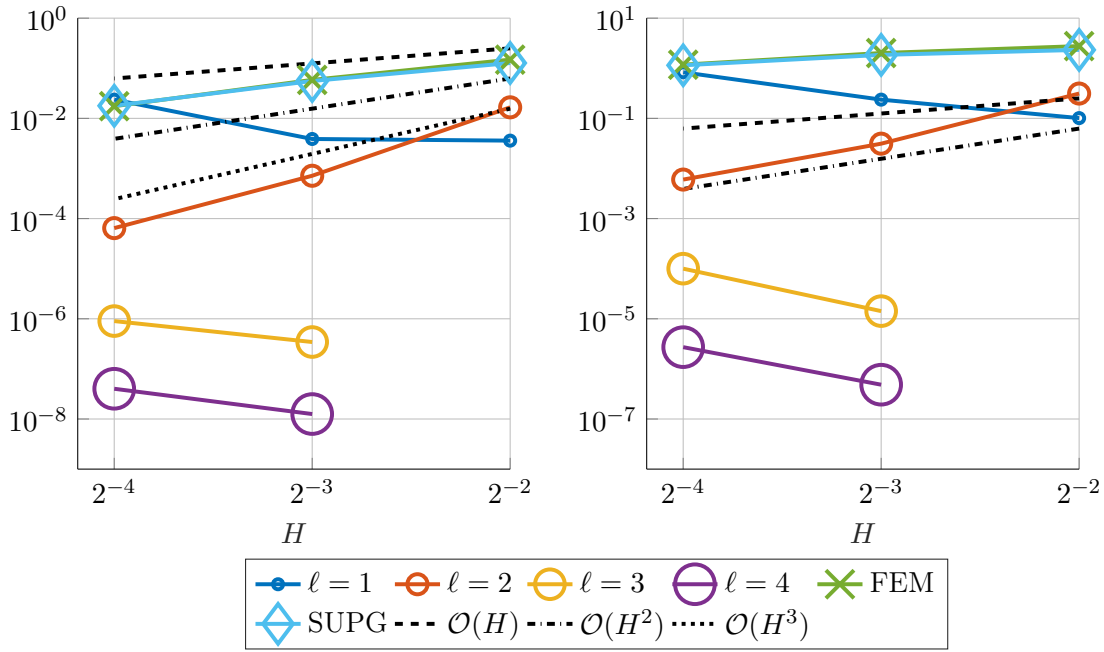
\Cref{con:sup_exp_decay}, for $d=3$, implies that the localization error behaves like $\exp(-C \ell^{\frac{3}{2}})$. \Cref{fig:convergence_3d_b} illustrates this super-exponential decay in the $\seminorm{\bullet}{V}$-norm.
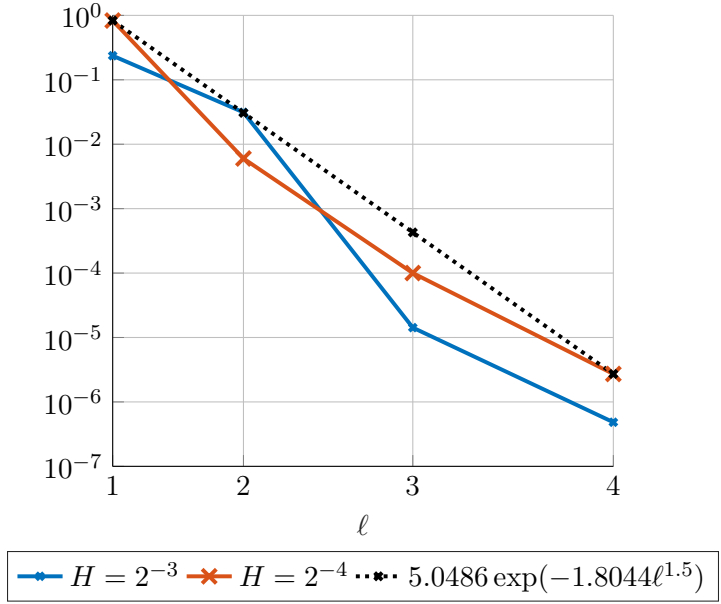
\begin{figure}
	\centering
%
%
\definecolor{mycolor1}{rgb}{0.00000,0.44700,0.74100}%
\definecolor{mycolor2}{rgb}{0.85000,0.32500,0.09800}%
\begin{tikzpicture}

\begin{axis}[%
width=6.656cm,
height=6cm,
at={(0cm,0cm)},
scale only axis,
xmin=1,
xmax=8,
xtick={1,2.82842712474619,5.19615242270663,8},
xticklabels={{1},{2},{3},{4}},
xlabel style={font=\color{white!15!black}},
xlabel={$\ell$},
ymode=log,
ymin=1e-07,
ymax=1,
yminorticks=false,
ylabel style={font=\color{white!15!black}},
axis background/.style={fill=white},
axis x line*=bottom,
axis y line*=left,
xmajorgrids,
ymajorgrids,
legend style={at={(0.5,-0.3)}, anchor=south, legend columns=3, legend cell align=left, align=left, draw=white!15!black}
]
\addplot [color=mycolor1, line width=1.5pt, mark size=1.9pt, mark=x, mark options={solid, mycolor1}]
  table[row sep=crcr]{%
1	0.236396769271956\\
2.82842712474619	0.0312581453439663\\
5.19615242270663	1.42287399389328e-05\\
8	4.83724707952136e-07\\
};
\addlegendentry{$H = 2^{-3}$}

\addplot [color=mycolor2, line width=1.5pt, mark size=3.8pt, mark=x, mark options={solid, mycolor2}]
  table[row sep=crcr]{%
1	0.830830305955967\\
2.82842712474619	0.00600399044100062\\
5.19615242270663	0.000100068241620616\\
8	2.71558464531153e-06\\
};
\addlegendentry{$H = 2^{-4}$}

\addplot [color=black, dotted, line width=1.5pt, mark=x, mark options={solid, black}]
  table[row sep=crcr]{%
1	0.830829678514806\\
2.82842712474619	0.0306649475664031\\
5.19615242270663	0.000427705362382283\\
8	2.71565432273585e-06\\
};
\addlegendentry{$5.0486\exp(-1.8044\ell^{1.5})$}

\end{axis}
\end{tikzpicture}%
	\caption{Super-exponential decay of the $\seminorm{\bullet}{V}$-norm in oversampling parameter $\ell$ for the three-dimensional experiment.}
	\label{fig:convergence_3d_b}
\end{figure}

\section{Concluding remarks and future developments}

We have presented a novel multi-scale method for convection-dominated problems. The method follows the LOD framework and employs a novel super-localization strategy. The resulting SLOD significantly improves previous attempts to tackle convection-dominated problems in the under-resolved regime of large mesh P\'eclet numbers. While previously the SLOD largely improved the performance of already very efficient methods for model diffusion and Helmholtz problems~\cite{Hauck-Peterseim,Freese-Hauck-Peterseim}, the present paper demonstrates the true potential of the super-localization idea to enlarge the class of problems tractable by multi-scale methods. The numerically observed $\varepsilon$-independent convergence in two as well as three-dimensional experiments is justified to some extent by numerical analysis involving a-priori and a-posteriori techniques. 

Among the many promising future research directions are parameterized elliptic multi-scale problems to be treated by combining the SLOD with model order reduction techniques, following the ideas in~\cite{Abdulle-Henning}.
An interesting and relevant application from the physical point of view are wave propagation and scattering problems in highly heterogeneous structures. For such target, the SLOD has been recently proposed in~\cite{Freese-Hauck-Peterseim}, and model order reduction techniques for the parametric-in-frequency problem have been recently presented (see, e.g., \cite{Bonizzoni-Nobile-Perugia,Bonizzoni-Nobile-Perugia-Pradovera-a,Bonizzoni-Nobile-Perugia-Pradovera-b,Bonizzoni-Pradovera-Ruggeri,Bonizzoni-Pradovera}). 
Moreover, the possible improvement of numerical stochastic homogenization methods \cite{Gallistl-Peterseim,Fischer-Gallistl-Peterseim,Feischl-Peterseim} and uncertainty quantification techniques~\cite{Bonizzoni-Buffa-Nobile,Bonizzoni-Nobile-a,Bonizzoni-Nobile-b,Bonizzoni-Nobile-c} will be analyzed.

\section*{Acknowledgment}
The authors gratefully acknowledge Gabriel Barrenechea for fruitful discussion on the stability properties of the convection-dominated boundary value problem.

\bibliographystyle{alpha}
\bibliography{bib}
\end{document}